\documentclass[a4paper]{amsart}
\usepackage{enumerate}
\usepackage{fourier}
\usepackage{hyperref}
\usepackage{mathrsfs}
\usepackage{tikz}
\usetikzlibrary{arrows,%
  decorations.markings,%
  matrix,%
  shapes}

\hypersetup{pdftitle={Higher n-angulations from local rings},%
  pdfauthor={Petter Andreas Bergh, Gustavo Jasso, and Marius Thaule},%
  pdfkeywords={Triangulated categories, $n$-angulated categories,
    $n$-exact categories, algebraic $n$-angulated categories}
}

\DeclareMathOperator{\der}{D}
\DeclareMathOperator{\Hom}{Hom}
\DeclareMathOperator{\id}{id}
\let\Im\relax
\DeclareMathOperator{\Im}{Im}

\let\mod\relax
\DeclareMathOperator{\mod}{mod}
\newcommand{\C}{\mathscr{C}}

\newcommand{\ceq}{\mathrel{\mathop:}=}
\newcommand{\diagram}[3]{\matrix (#1) [matrix of math nodes,row
  sep={#2},column sep={#3},text height=1.5ex,text depth=0.25ex]}
\newcommand{\E}{\mathscr{E}}
\newcommand{\m}{\mathfrak{m}}
\newcommand{\nang}{\mathscr{N}}
\newcommand{\sE}{\underline{\E}}

\newcommand{\X}{\mathcal{X}}
\newcommand{\Z}{\mathbb{Z}}

\theoremstyle{definition}
\newtheorem*{definition}{Definition}
\newtheorem*{notation}{Notation}
\newtheorem*{example}{Example}
\newtheorem{remark}[subsection]{Remark}

\theoremstyle{plain}
\newtheorem{theorem}[subsection]{Theorem}
\newtheorem{lemma}[subsection]{Lemma}
\newtheorem{proposition}[subsection]{Proposition}
\newtheorem{corollary}[subsection]{Corollary}

\title{Higher $n$-angulations from local rings}
\author{Petter Andreas Bergh}
\address[Bergh]{Department of Mathematical Sciences, NTNU, N-7491 Trondheim,
  Norway}
\email{bergh@math.ntnu.no}
\urladdr{http://www.math.ntnu.no/~bergh/}
\author{Gustavo Jasso}
\address[Jasso]{Mathematik Zentrum, University of Bonn, Endenicher Allee 60,
  D-53115 Bonn, Germany}
\email{gjasso@math.uni-bonn.de}
\urladdr{https://gustavo.jasso.info}
\author{Marius Thaule}
\address[Thaule]{Department of Mathematical Sciences, NTNU, N-7491
  Trondheim, Norway}
\email{mariusth@math.ntnu.no}
\urladdr{http://www.math.ntnu.no/~mariusth/}
\date{\today}


\begin{document}

\begin{abstract}
  We show that the category of finitely generated free modules over
  certain local rings is $n$-angulated for every $n \geq 3$.  In fact,
  we construct several classes of $n$-angles, parametrized by
  equivalence classes of units in the local rings. Finally, we show
  that for odd values of $n$ some of these $n$-angulated categories
  are not algebraic.
\end{abstract}

\maketitle

\section{Introduction}
\label{sec:intro}

In \cite{MSS}, Muro, Schwede and Strickland proved that the category
of finitely generated free modules over a local commutative ring $R$
with maximal ideal $\m = (2) \neq 0$ and $\m^2 = 0$ is triangulated
with $\Sigma = \id$, where the collection of distinguished triangles
is given by the contractible triangles, the triangle
\begin{equation*}
  R \xrightarrow{2} R \xrightarrow{2} R \xrightarrow{2} R
\end{equation*}
and direct sums of these.  Moreover, they proved that this particular
triangulated category is neither algebraic nor topological.  For a
discussion on algebraic and topological triangulated categories, see
e.g.\ \cite{S1, S2, S3}.

Geiss, Keller and Oppermann recently introduced in \cite{GKO} ``higher
dimensional'' analogues of triangulated categories, called
\emph{$n$-angulated categories}, and showed that certain cluster
tilting subcategories of triangulated categories give rise to
$n$-angulated categories.  For $n = 3$, an $n$-angulated category is
the same as a classical triangulated category. The theory of
$n$-angulated categories has since been developed further: an
equivalent set of axioms was given in \cite{BT1}, and a generalization
of Thomason's classification theorem for triangulated subcategories
was proved in \cite{BT2}.

In this paper, we show that for a commutative local ring with
principal maximal ideal squaring to zero, the category of finitely
generated free modules is $n$-angulated for every $n \geq 3$. More
precisely, for such a ring and any integer $n \geq 3$, we construct a
class of $n$-angles in the category of finitely generated free
modules, drawing upon the work of Muro, Schwede and Strickland
mentioned above. In fact, we construct several classes of $n$-angles
on the same suspended category, parametrized in terms of a certain
equivalence relation on the set of units in the local ring. The
different classes arise from global automorphisms on the underlying
category, introduced by Balmer in \cite{Balmer}. Thus we obtain
examples where there are infinitely many classes of $n$-angles,
finitely many classes, and only one class.

In \cite{J}, the second author introduced the notion of algebraic
$n$-angulated categories. In analogy with the case of triangulated
categories, these are the $n$-angulated categories that are equivalent
to stable categories of ``higher dimensional'' Frobenius exact
categories. We show that for certain local rings, the $n$-angulated
categories we construct are not algebraic. As mentioned above, this
was done in \cite{MSS} in the triangulated case, but there it was also
proved that the triangulated category in question is not even
topological. However, at the time of writing, there is no notion of
``topological'' $n$-angulated categories to compare with in the higher
setting.

\sloppy The paper is organized as follows: in Section \ref{sec:pre} we
recall the definition of an $n$-angulated category and prove some
elementary results on homotopic $n$-$\Sigma$-sequences, and in Section
\ref{sec:example} we prove our main result. Then in Section
\ref{sec:properties}, we show that there are in general several
$n$-angulations on the underlying suspended category. Finally, in
Section \ref{sec:algebraic}, we show that for certain local rings, the
$n$-angulated categories we construct are not algebraic (for odd
values of $n$).

\section{Preliminaries}
\label{sec:pre}

We recall the definition of an $n$-angulated category from \cite{GKO}.
Let $\C$ be an additive category with an automorphism $\Sigma \colon
\C \to \C$, and $n$ an integer greater than or equal to three.  An
\emph{$n$-$\Sigma$-sequence} in $\C$ is a sequence
\begin{equation*}
  A_1 \xrightarrow{\alpha_1} A_2 \xrightarrow{\alpha_2} \cdots
  \xrightarrow{\alpha_{n - 1}} A_n \xrightarrow{\alpha_n} \Sigma A_1 
\end{equation*}
of objects and morphisms in $\C$.  We shall often denote such
sequences by $A_\bullet, B_\bullet$ etc.  Its left and right
\emph{rotations} are the two $n$-$\Sigma$-sequences
\begin{equation*}
  A_2 \xrightarrow{\alpha_2} A_3 \xrightarrow{\alpha_3} \cdots
  \xrightarrow{\alpha_n} \Sigma A_1 \xrightarrow{(-1)^n
    \Sigma\alpha_1} \Sigma A_2
\end{equation*}
and
\begin{equation*}
  \Sigma^{-1}A_n \xrightarrow{(-1)^n \Sigma^{-1} \alpha_n} A_1
  \xrightarrow{\alpha_1} \cdots \xrightarrow{\alpha_{n - 2}} A_{n - 1}
  \xrightarrow{\alpha_{n - 1}} A_n
\end{equation*} 
respectively, and it is \emph{exact} if the induced sequence
\begin{align*}
  \cdots \to \Hom_{\C}(B,A_1) & \xrightarrow{(\alpha_1)_*}
  \Hom_{\C}(B,A_2) \xrightarrow{(\alpha_2)_*} \cdots\\
  & \cdots \xrightarrow{(\alpha_{n - 1})_*} \Hom_{\C}(B,A_n)
  \xrightarrow{(\alpha_n)_*} \Hom_{\C}(B, \Sigma A_1) \to \cdots
\end{align*}
of abelian groups is exact for every object $B \in \C$. A
\emph{trivial} $n$-$\Sigma$-sequence is a sequence of the form
\begin{equation*}
  A \xrightarrow{1} A \to 0 \to \cdots \to 0 \to \Sigma A
\end{equation*}
or any of its rotations.  

A \emph{morphism} $A_{\bullet} \xrightarrow{\varphi} B_{\bullet}$ of
$n$-$\Sigma$-sequences is a sequence $\varphi =
(\varphi_1,\varphi_2,\ldots,\varphi_n)$ of morphisms in $\C$ such that
the diagram
\begin{center}
  \begin{tikzpicture}
    \diagram{d}{2.5em}{2.5em}{
      A_1 & A_2 & \cdots & A_n & \Sigma A_1\\
      B_1 & B_2 & \cdots & B_n & \Sigma B_1\\
    };

    \path[->,midway,font=\scriptsize]
      (d-1-1) edge node[above]{$\alpha_1$} (d-1-2)
                   edge node[right]{$\varphi_1$} (d-2-1)
      (d-1-2) edge node[above]{$\alpha_2$} (d-1-3)
                   edge node[right]{$\varphi_2$} (d-2-2)
      (d-1-3) edge node[above]{$\alpha_{n - 1}$} (d-1-4)
      (d-1-4) edge node[above]{$\alpha_n$} (d-1-5)
                   edge node[right]{$\varphi_n$} (d-2-4)
      (d-1-5) edge node[right]{$\Sigma\varphi_{1}$} (d-2-5)
      (d-2-1) edge node[above]{$\beta_1$} (d-2-2)
      (d-2-2) edge node[above]{$\beta_2$} (d-2-3)
      (d-2-3) edge node[above]{$\beta_{n - 1}$} (d-2-4)
      (d-2-4) edge node[above]{$\beta_n$} (d-2-5);
  \end{tikzpicture}
\end{center}
commutes.  It is an \emph{isomorphism} if
$\varphi_1,\varphi_2,\ldots,\varphi_n$ are all isomorphisms in $\C$,
and a \emph{weak isomorphism} if $\varphi_i$ and $\varphi_{i+1}$ are
isomorphisms for some $i$ (with $\varphi_{n+1} \ceq \Sigma
\varphi_1$).

The category $\C$ is \emph{pre-$n$-angulated} if there exists a
collection $\nang$ of $n$-$\Sigma$-sequences satisfying the following
three axioms:
\begin{itemize}
\item[{\textbf{(N1)}}]
  \begin{itemize}
  \item[(a)] $\nang$ is closed under direct sums, direct summands and
    isomorphisms of $n$-$\Sigma$-sequences.
  \item[(b)] For all $A \in \C$, the trivial $n$-$\Sigma$-sequence
   \begin{equation*}
     A \xrightarrow{1} A \to 0 \to \cdots \to 0 \to \Sigma A
   \end{equation*}
   belongs to $\nang$.
 \item[(c)] For each morphism $\alpha \colon A_1 \to A_2$ in $\C$,
   there exists an $n$-$\Sigma$-sequence in $\nang$ whose first
   morphism is $\alpha$.
 \end{itemize}
\item[{\textbf{(N2)}}] An $n$-$\Sigma$-sequence belongs to $\nang$ if
  and only if its left rotation belongs to $\nang$.
\item[{\textbf{(N3)}}] Each commutative diagram
  \begin{center}
    \begin{tikzpicture}
      \diagram{d}{2.5em}{2.5em}{
        A_1 & A_2 & A_3 & \cdots & A_n & \Sigma A_1\\
        B_1 & B_2 & B_3 & \cdots & B_n & \Sigma B_1\\
      };

      \path[->,midway,font=\scriptsize]
        (d-1-1) edge node[above]{$\alpha_1$} (d-1-2)
                     edge node[right]{$\varphi_1$} (d-2-1)
        (d-1-2) edge node[above]{$\alpha_2$} (d-1-3)
                     edge node[right]{$\varphi_2$} (d-2-2)
        (d-1-3) edge node[above]{$\alpha_3$} (d-1-4)
                     edge[densely dashed] node[right]{$\varphi_3$} (d-2-3)
        (d-1-4) edge node[above]{$\alpha_{n - 1}$} (d-1-5)
        (d-1-5) edge node[above]{$\alpha_n$} (d-1-6)
                     edge[densely dashed] node[right]{$\varphi_n$} (d-2-5)
        (d-1-6) edge node[right]{$\Sigma\varphi_{1}$} (d-2-6)
        (d-2-1) edge node[above]{$\beta_1$} (d-2-2)
        (d-2-2) edge node[above]{$\beta_2$} (d-2-3)
        (d-2-3) edge node[above]{$\beta_3$} (d-2-4)
        (d-2-4) edge node[above]{$\beta_{n - 1}$} (d-2-5)
        (d-2-5) edge node[above]{$\beta_n$} (d-2-6);
    \end{tikzpicture}
  \end{center}
  with rows in $\nang$ can be completed to a morphism of
  $n$-$\Sigma$-sequences.
\end{itemize}

The collection $\nang$ satisfying the above three axioms is a
\emph{pre-$n$-angulation} of the category $\C$ (relative to the
automorphism $\Sigma$), and the $n$-$\Sigma$-sequences in $\nang$ are
called \emph{$n$-angles}.  If, in addition, the collection $\nang$
satisfies the following axiom, then it is an \emph{$n$-angulation} of
$\C$, and the category is \emph{$n$-angulated}:

\begin{itemize}
\item[{\textbf{(N4)}}] In the situation of (N3), the morphisms
  $\varphi_3,\varphi_4,\ldots,\varphi_n$ can be chosen such that the
  mapping cone
  \begin{equation*}
    A_2 \oplus B_1 \xrightarrow{\left[
        \begin{smallmatrix}
          -\alpha_2 & 0\\
          \hfill \varphi_2 & \beta_1
        \end{smallmatrix}
      \right]} A_3 \oplus B_2 \xrightarrow{\left[
        \begin{smallmatrix}
          -\alpha_3 & 0\\
          \hfill \varphi_3 & \beta_2
        \end{smallmatrix}
      \right]} \cdots \xrightarrow{\left[
        \begin{smallmatrix} 
          -\alpha_n & 0\\
          \hfill \varphi_n & \beta_{n - 1}
        \end{smallmatrix}
      \right]} \Sigma A_1 \oplus B_n \xrightarrow{\left[
       \begin{smallmatrix}
          -\Sigma \alpha_1 & 0\\
          \hfill \Sigma \varphi_1 & \beta_n
        \end{smallmatrix}
      \right]} \Sigma A_2 \oplus \Sigma B_1
  \end{equation*}
belongs to $\nang$.
\end{itemize}

\sloppy To be precise, one should include all the data when referring
to a (pre-)$n$-angulated category, and therefore write $(\C, \Sigma,
\nang)$. Note that by \cite[Proposition 1.5(c)]{GKO}, if $(\C, \Sigma,
\nang_1)$ and $(\C, \Sigma, \nang_2)$ are pre-$n$-angulated categories
(with the same underlying category $\C$ and automorphism $\Sigma$)
with $\nang_1 \subseteq \nang_2$, then the pre-$n$-angulations must
actually coincide, i.e.\ $\nang_1 = \nang_2$.

\begin{definition}
  \label{def:homotopy}
  Let $A_\bullet$ and $B_\bullet$ be two $n$-$\Sigma$-sequences and
  $\varphi$ and $\psi$ two morphisms from $A_\bullet$ to $B_\bullet$.
  A \emph{homotopy} $\Theta$ from $\varphi$ to $\psi$ is given by
  diagonal morphisms $\Theta_i$
  \begin{center}
    \begin{tikzpicture}
      \diagram{d}{3.5em}{4em}{
        A_1 & A_2 & A_3 & \cdots & A_n & \Sigma A_1\\
        B_1 & B_2 & B_3 & \cdots & B_n & \Sigma B_1\\
      };

      \path[->,midway,font=\scriptsize]
        (d-1-1) edge node[above]{$\alpha_1$} (d-1-2)
        ([xshift=-1mm] d-1-1.south) edge node[left]{$\varphi_1$}
        ([xshift=-1mm] d-2-1.north)
        ([xshift=1mm] d-1-1.south) edge node[right]{$\psi_1$}
        ([xshift=1mm] d-2-1.north)
        (d-1-2) edge node[above]{$\alpha_2$} (d-1-3)
                     edge node[above left]{$\Theta_1$} (d-2-1)
        ([xshift=-1mm] d-1-2.south) edge node[left]{$\varphi_2$}
        ([xshift=-1mm] d-2-2.north)
        ([xshift=1mm] d-1-2.south) edge node[right]{$\psi_2$}
        ([xshift=1mm] d-2-2.north)
        (d-1-3) edge node[above]{$\alpha_3$} (d-1-4)
                     edge node[above left]{$\Theta_2$} (d-2-2)
        ([xshift=-1mm] d-1-3.south) edge node[left]{$\varphi_3$}
        ([xshift=-1mm] d-2-3.north)
        ([xshift=1mm] d-1-3.south) edge node[right]{$\psi_3$}
        ([xshift=1mm] d-2-3.north)
        (d-1-4) edge node[above]{$\alpha_{n - 1}$} (d-1-5)
        (d-1-5) edge node[above]{$\alpha_n$} (d-1-6)
        ([xshift=-1mm] d-1-5.south) edge node[left]{$\varphi_n$}
        ([xshift=-1mm] d-2-5.north)
        ([xshift=1mm] d-1-5.south) edge node[right]{$\psi_n$}
        ([xshift=1mm] d-2-5.north)
        (d-1-6) edge node[above left]{$\Theta_n$} (d-2-5)
        ([xshift=-1mm] d-1-6.south) edge node[left]{$\Sigma\varphi_1$}
        ([xshift=-1mm] d-2-6.north)
        ([xshift=1mm] d-1-6.south) edge node[right]{$\Sigma\psi_1$}
        ([xshift=1mm] d-2-6.north)
        (d-2-1) edge node[above]{$\beta_1$} (d-2-2)
        (d-2-2) edge node[above]{$\beta_2$} (d-2-3)
        (d-2-3) edge node[above]{$\beta_3$} (d-2-4)
        (d-2-4) edge node[above]{$\beta_{n - 1}$} (d-2-5)
        (d-2-5) edge node[above]{$\beta_n$} (d-2-6);
    \end{tikzpicture}
  \end{center}
  such that
  \begin{equation*}
    \varphi_i - \psi_i = \Theta_i \circ \alpha_i + \beta_{i - 1} \circ
    \Theta_{i - 1} \quad \text{for} \quad i = 2,3,\ldots,n,
  \end{equation*}
  and
  \begin{equation*}
    \Sigma \varphi_1 - \Sigma \psi_1 = \Sigma\Theta_1 \circ \Sigma
    \alpha_1 + \beta_n \circ \Theta_n.
  \end{equation*}
  In this case, we say that $\varphi$ and $\psi$ are
  \emph{homotopic}. A morphism homotopic to the zero morphism is
  called \emph{nullhomotopic}; an $n$-$\Sigma$-sequence is
  \emph{contractible} if its identity morphism is nullhomotopic.
\end{definition}

\begin{lemma}\label{lem:homotopic}
  \emph{(1)} Any morphism to or from a contractible
  $n$-$\Sigma$-sequence is nullhomotopic.

  \emph{(2)} Homotopic morphisms of $n$-$\Sigma$-sequences have
  isomorphic mapping cones.
\end{lemma}  

\begin{proof}
  (1) This is standard: if $A_\bullet \xrightarrow{f} B_\bullet,
  B_\bullet \xrightarrow{\varphi} C_\bullet, B_\bullet
  \xrightarrow{\psi} C_\bullet, C_\bullet \xrightarrow{g} D_\bullet$
  are morphisms of $n$-$\Sigma$-sequences with $\varphi$ and $\psi$
  homotopic, then $g \circ \varphi \circ f$ and $g \circ \psi \circ f$
  are homotopic.

  (2) We use the notation from the definition above. The diagram
  \begin{center}
    \begin{tikzpicture}
      \diagram{d}{3em}{4.75em}{ 
        A_2 \oplus B_1 & A_3 \oplus B_2 & \cdots & \Sigma A_1 \oplus
        B_n & \Sigma A_2 \oplus \Sigma B_1\\
        A_2 \oplus B_1 & A_3 \oplus B_2 & \cdots & \Sigma A_1 \oplus
        B_n & \Sigma A_2 \oplus \Sigma B_1\\
        A_2 \oplus B_1 & A_3 \oplus B_2 & \cdots & \Sigma A_1 \oplus
        B_n & \Sigma A_2 \oplus \Sigma B_1\\
      };

      \path[->,midway,font=\scriptsize]
        (d-1-1) edge node[above]{$\left[
            \begin{smallmatrix}
              -\alpha_2 & 0\\
              \hfill \varphi_2 & \beta_1
            \end{smallmatrix}
          \right]$} (d-1-2)
                edge node[right]{$\left[
                    \begin{smallmatrix}
                      1 & 0\\
                      \Theta_1 & 1
                    \end{smallmatrix}
                  \right]$} (d-2-1)
        (d-1-2) edge node[above]{$\left[
            \begin{smallmatrix}
              -\alpha_3 & 0\\
              \hfill \varphi_3 & \beta_2
            \end{smallmatrix}
          \right]$} (d-1-3)
                edge node[right]{$\left[
                    \begin{smallmatrix}
                      1 & 0\\
                      \Theta_2 & 1
                    \end{smallmatrix}
                  \right]$} (d-2-2)
        (d-1-3) edge node[above]{$\left[
            \begin{smallmatrix}
              -\alpha_n & 0\\
              \hfill \varphi_n & \beta_{n - 1}
            \end{smallmatrix}
          \right]$} (d-1-4)
        (d-1-4) edge node[above]{$\left[
            \begin{smallmatrix}
              -\Sigma\alpha_1 & 0\\
              \hfill \Sigma\varphi_1 & \beta_n
            \end{smallmatrix}
          \right]$} (d-1-5)
                edge node[right]{$\left[
                    \begin{smallmatrix}
                      1 & 0\\
                      \Theta_n & 1
                    \end{smallmatrix}
                  \right]$} (d-2-4)
        (d-1-5) edge node[right]{$\left[
            \begin{smallmatrix}
              1 & 0\\
              \Sigma\Theta_1 & 0
            \end{smallmatrix}
          \right]$} (d-2-5)
        (d-2-1) edge node[above]{$\left[
            \begin{smallmatrix}
              -\alpha_2 & 0\\
              \hfill \psi_2 & \beta_1
            \end{smallmatrix}
          \right]$} (d-2-2)
                edge node[right]{$\left[
                    \begin{smallmatrix}
                      1 & 0\\
                      -\Theta_1 & 1
                    \end{smallmatrix}
                  \right]$} (d-3-1)
        (d-2-2) edge node[above]{$\left[
            \begin{smallmatrix}
              -\alpha_3 & 0\\
              \hfill \psi_3 & \beta_2
            \end{smallmatrix}
          \right]$} (d-2-3)
                edge node[right]{$\left[
                    \begin{smallmatrix}
                      1 & 0\\
                      -\Theta_2 & 1
                    \end{smallmatrix}
                  \right]$} (d-3-2)
        (d-2-3) edge node[above]{$\left[
            \begin{smallmatrix}
              -\alpha_n & 0\\
              \hfill \psi_n & \beta_{n - 1}
            \end{smallmatrix}
          \right]$} (d-2-4)
        (d-2-4) edge node[above]{$\left[
            \begin{smallmatrix}
              -\Sigma\alpha_1 & 0\\
              \hfill \Sigma\psi_1 & \beta_n
            \end{smallmatrix}
          \right]$} (d-2-5)
                edge node[right]{$\left[
                    \begin{smallmatrix}
                      1 & 0\\
                      -\Theta_n & 1
                    \end{smallmatrix}
                  \right]$} (d-3-4)
        (d-2-5) edge node[right]{$\left[
            \begin{smallmatrix}
              1 & 0\\
              -\Sigma\Theta_1 & 0
            \end{smallmatrix}
          \right]$} (d-3-5)
        (d-3-1) edge node[above]{$\left[
            \begin{smallmatrix}
              -\alpha_2 & 0\\
              \hfill \varphi_2 & \beta_1
            \end{smallmatrix}
          \right]$} (d-3-2)
        (d-3-2) edge node[above]{$\left[
            \begin{smallmatrix}
              -\alpha_3 & 0\\
              \hfill \varphi_3 & \beta_2
            \end{smallmatrix}
          \right]$} (d-3-3)
        (d-3-3) edge node[above]{$\left[
            \begin{smallmatrix}
              -\alpha_n & 0\\
              \hfill \varphi_n & \beta_{n - 1}
            \end{smallmatrix}
          \right]$} (d-3-4)
        (d-3-4) edge node[above]{$\left[
            \begin{smallmatrix}
              -\Sigma\alpha_1 & 0\\
              \hfill \Sigma\varphi_1 & \beta_n
            \end{smallmatrix}
          \right]$} (d-3-5);
    \end{tikzpicture}
  \end{center}
  displays inverse isomorphisms between the two mapping cones of
  $\varphi$ and $\psi$.
\end{proof}

By \cite[Proposition 1.5(a)]{GKO}, any $n$-angle 
\begin{equation*}
  A_1 \xrightarrow{\alpha_1} A_2 \xrightarrow{\alpha_2} \cdots
  \xrightarrow{\alpha_{n - 1}} A_n \xrightarrow{\alpha_n} \Sigma A_1 
\end{equation*}
in a pre-$n$-angulated category $\C$ is necessarily exact, so that
when applying $\Hom_{\C}(B,-)$ for any object $B$, the result is a
long exact sequence of abelian groups. Consequently, the compositions
\begin{equation*}
  \alpha_2 \circ \alpha_1, \hspace{2mm} \alpha_3 \circ \alpha_2,\dots,
  \alpha_n \circ \alpha_{n - 1}, \hspace{2mm} \Sigma \alpha_1 \circ
  \alpha_n
\end{equation*}
are all zero morphisms in $\C$. We call an $n$-$\Sigma$-sequence with
this last property a \emph{candidate $n$-angle}. Thus, exact
$n$-$\Sigma$-sequences are candidate $n$-angles. The converse is of
course not true in general, but the following result shows that it
holds for contractible candidate $n$-angles. Moreover, the result
shows that contractible candidate $n$-angles and mapping cones of
isomorphisms automatically belong to \emph{any} pre-$n$-angulation of
the category. For triangulated categories, these results were proved
in \cite[Lemma 1.3.7 and Proposition 1.3.8]{N} and \cite[pp.\
231--232]{N1}.

\begin{lemma}
  \label{lem:contractible}
  If $(\C, \Sigma, \nang)$ is a pre-$n$-angulated category and
  $A_\bullet$ a candidate $n$-angle in $\C$, then the following hold.
  
  (1) If $A_\bullet$ is contractible, then it is exact and belongs to $\nang$.
  
  (2) If $\varphi \colon A_\bullet \to A_\bullet$ is an isomorphism of
  $n$-$\Sigma$-sequences, then its mapping cone is contractible (and
  therefore belongs to $\nang$ by (1)).
\end{lemma} 

\begin{proof}
  (1) Since $A_\bullet$ is contractible, its identity morphism is
  homotopic to the zero morphism. Thus there exist diagonal morphisms
  \begin{center}
    \begin{tikzpicture}
      \diagram{d}{3em}{4em}{
        A_1 & A_2 & A_3 & \cdots & A_n & \Sigma A_1\\
        A_1 & A_2 & A_3 & \cdots & A_n & \Sigma A_1\\
      };

      \path[->,midway,font=\scriptsize]
        (d-1-1) edge node[above]{$\alpha_1$} (d-1-2)
                     edge node[left]{$1$} (d-2-1)
        (d-1-2) edge node[above]{$\alpha_2$} (d-1-3)
                     edge node[above left]{$\Theta_1$} (d-2-1)
                     edge node[left]{$1$} (d-2-2.north)
        (d-1-3) edge node[above]{$\alpha_3$} (d-1-4)
                     edge node[above left]{$\Theta_2$} (d-2-2)
                     edge node[left]{$1$} (d-2-3)
        (d-1-4) edge node[above]{$\alpha_{n - 1}$} (d-1-5)
        (d-1-5) edge node[above]{$\alpha_n$} (d-1-6)
                     edge node[left]{$1$} (d-2-5)
        (d-1-6) edge node[above left]{$\Theta_n$} (d-2-5)
                     edge node[left]{$1$} (d-2-6)
        (d-2-1) edge node[above]{$\alpha_1$} (d-2-2)
        (d-2-2) edge node[above]{$\alpha_2$} (d-2-3)
        (d-2-3) edge node[above]{$\alpha_3$} (d-2-4)
        (d-2-4) edge node[above]{$\alpha_{n - 1}$} (d-2-5)
        (d-2-5) edge node[above]{$\alpha_n$} (d-2-6);
    \end{tikzpicture}
  \end{center}
  in $\C$ satisfying
  \begin{equation*}
    1_{A_i} = \Theta_i \circ \alpha_i + \alpha_{i - 1} \circ
    \Theta_{i - 1} \quad \text{for} \quad i = 2,3,\ldots,n,
  \end{equation*}
  and
  \begin{equation*}
    1_{\Sigma A_1} = \Sigma\Theta_1 \circ \Sigma
    \alpha_1 + \alpha_n \circ \Theta_n.
  \end{equation*}
  If $B$ is any object in $\C$, then applying $\Hom_{\C}(B,-)$ to
  $A_\bullet$ gives a complex $\Hom_{\C}(B, A_\bullet )$ of abelian
  groups, since $A_\bullet$ is a candidate $n$-angle. Moreover, this
  complex is contractible, as is seen directly by applying
  $\Hom_{\C}(B,-)$ to the contracting homotopy above. The complex
  $\Hom_{\C}(B, A_\bullet )$ is therefore exact (every contractible
  complex of abelian groups is exact), hence $A_\bullet$ is exact.

  Next, we show that $A_\bullet$ is an $n$-angle, i.e.\ that
  $A_\bullet \in \nang$. For $n=3$, that is, when $\C$ is a
  pre-triangulated category, this is just \cite[Proposition
  1.3.8]{N}. Therefore we may assume that $n \geq 4$.

  By axiom (N1)(c), there exists an $n$-angle
  \begin{equation*}
    B_\bullet\colon \quad A_1 \xrightarrow{\alpha_1} A_2
    \xrightarrow{\beta_2} B_3 \xrightarrow{\beta_3} \cdots
    \xrightarrow{\beta_{n - 1}} B_n \xrightarrow{\beta_n} \Sigma A_1 
  \end{equation*}
  whose first morphism is $\alpha_1$. Using the contracting homotopy
  above, we shall complete the solid part of the diagram
  \begin{center}
    \begin{tikzpicture}
      \diagram{d}{2.5em}{2.5em}{ 
        A_1 & A_2 & A_3 & \cdots & A_n & \Sigma A_1\\
        A_1 & A_2 & B_3 & \cdots & B_n & \Sigma A_1\\
      };

      \path[->,midway,font=\scriptsize]
        (d-1-1) edge node[above]{$\alpha_1$} (d-1-2)
                     edge node[right]{$1$} (d-2-1)
        (d-1-2) edge node[above]{$\alpha_2$} (d-1-3)
                     edge node[right]{$1$} (d-2-2)
        (d-1-3) edge node[above]{$\alpha_3$} (d-1-4)
                     edge[densely dashed] node[right]{$\varphi_3$} (d-2-3)
        (d-1-4) edge node[above]{$\alpha_{n - 1}$} (d-1-5)
        (d-1-5) edge node[above]{$\alpha_n$} (d-1-6)
                     edge[densely dashed] node[right]{$\varphi_n$} (d-2-5)
        (d-1-6) edge node[right]{$1$} (d-2-6.north)
        (d-2-1) edge node[above]{$\alpha_1$} (d-2-2)
        (d-2-2) edge node[above]{$\beta_2$} (d-2-3)
        (d-2-3) edge node[above]{$\beta_3$} (d-2-4)
        (d-2-4) edge node[above]{$\beta_{n - 1}$} (d-2-5)
        (d-2-5) edge node[above]{$\beta_n$} (d-2-6);
    \end{tikzpicture}
  \end{center}
  to a morphism $(1,1, \varphi_3, \dots, \varphi_n)$ of
  $n$-$\Sigma$-sequences, a weak isomorphism. Then since $B_\bullet$
  is an $n$-angle and $A_\bullet$ is exact, the latter must also be an
  $n$-angle by \cite[Lemma 1.4]{GKO}.

  We start with the two morphisms $\varphi_3$ and $\varphi_n$. Define
  $\varphi_3$ by
  \begin{equation*}
    \varphi_3 \ceq \beta_2 \circ \Theta_2.
  \end{equation*}
  Then 
  \begin{equation*}
    \varphi_3 \circ \alpha_2 = \beta_2 \circ \Theta_2 \circ \alpha_2 =
    \beta_2 \circ ( 1_{A_2} - \alpha_1 \circ \Theta_1) = \beta_2
  \end{equation*}
  since $\beta_2 \circ \alpha_1 =0$, hence the second square from the
  left commutes. To define $\varphi_n$, we use the exactness of
  \begin{equation*}
    \Hom_{\C}( \Sigma A_1, B_n) \xrightarrow{(\beta_n)_*} \Hom_{\C}(
    \Sigma A_1, \Sigma A_1) \xrightarrow{(\Sigma \alpha_1)_*}
    \Hom_{\C}( \Sigma A_1, \Sigma A_2)
  \end{equation*}
  which follows from \cite[Proposition 1.5(a)]{GKO} and the fact that
  $B_\bullet$ is an $n$-angle. Consider the morphism $\alpha_n \circ
  \Theta_n$ in $\Hom_{\C}( \Sigma A_1, \Sigma A_1)$. When applying
  $(\Sigma \alpha_1)_*$ to this morphism, the result is zero since
  $\Sigma \alpha_1 \circ \alpha_n =0$. Therefore, by the exactness of
  the above sequence, there exists a morphism $\Theta \in \Hom_{\C}(
  \Sigma A_1, B_n)$ with $\beta_n \circ \Theta = \alpha_n \circ
  \Theta_n$. Now define $\varphi_n$ by
  \begin{equation*}
    \varphi_n \ceq \Theta \circ \alpha_n.
  \end{equation*}
  Then
  \begin{equation*}
    \beta_n \circ \varphi_n = \beta_n \circ \Theta \circ \alpha_n =
    \alpha_n \circ \Theta_n \circ \alpha_n = (1_{\Sigma A_1} - \Sigma
    \Theta_1 \circ \Sigma \alpha_1 ) \circ \alpha_n = \alpha_n
  \end{equation*}
  since $\Sigma \alpha_1 \circ \alpha_n$, hence the rightmost square
  commutes.

  Since $\beta_3 \circ \beta_2 =0$, it follows from the definition of
  $\varphi_3$ that $\beta_3 \circ \varphi_3 =0$. Similarly, since
  $\alpha_n \circ \alpha_{n - 1} =0$, it follows from the definition of
  $\varphi_n$ that $\varphi_n \circ \alpha_{n - 1} =0$. Therefore, if
  $n=4$, we have obtained a morphism $(1,1, \varphi_3, \varphi_4)$ of
  $4$-$\Sigma$-sequences. If $n \geq 5$, we let $\varphi_4, \dots,
  \varphi_{n - 1}$ all be zero morphisms: in this case $(1,1, \varphi_3,
  0, \dots, 0, \varphi_n)$ is a morphism of $n$-$\Sigma$-sequences.
  
  (2) The mapping cone of $\varphi$ is the $n$-$\Sigma$-sequence
  \begin{equation*}
    A_2 \oplus A_1 \xrightarrow{\left[
        \begin{smallmatrix}
          -\alpha_2 & 0\\
          \hfill \varphi_2 & \alpha_1
        \end{smallmatrix}
      \right]} A_3 \oplus A_2 \xrightarrow{\left[
        \begin{smallmatrix}
          -\alpha_3 & 0\\
          \hfill \varphi_3 & \alpha_2
        \end{smallmatrix}
      \right]} \cdots \xrightarrow{\left[
        \begin{smallmatrix} 
          -\alpha_n & 0\\
          \hfill \varphi_n & \alpha_{n - 1}
        \end{smallmatrix}
      \right]} \Sigma A_1 \oplus A_n \xrightarrow{\left[
       \begin{smallmatrix}
          -\Sigma \alpha_1 & 0\\
          \hfill \Sigma \varphi_1 & \alpha_n
        \end{smallmatrix}
      \right]} \Sigma A_2 \oplus \Sigma A_1
  \end{equation*}
  which is easily seen to be a candidate $n$-angle. The diagram
  \begin{center}
    \begin{tikzpicture}
      \diagram{d}{4em}{5em}{
        A_2 \oplus A_1 & A_3 \oplus A_2 & \cdots & \Sigma A_1 \oplus
        A_n & \Sigma A_2 \oplus \Sigma A_1\\
        A_2 \oplus A_1 & A_3 \oplus A_2 & \cdots & \Sigma A_1 \oplus
        A_n & \Sigma A_2 \oplus \Sigma A_1\\
      };

      \path[->,midway,font=\scriptsize]
        (d-1-1) edge node[above]{$\left[
            \begin{smallmatrix}
              - \alpha_2 & 0\\
              \hfill \varphi_2 & \alpha_1
            \end{smallmatrix}
          \right]$} (d-1-2)
                edge node[right]{$\left[
                    \begin{smallmatrix}
                      1 & 0\\
                      0 & 1
                    \end{smallmatrix}
                  \right]$} (d-2-1)
        (d-1-2) edge node[above]{$\left[
            \begin{smallmatrix}
              - \alpha_3 & 0\\
              \hfill \varphi_3 & \alpha_2
            \end{smallmatrix}
          \right]$} (d-1-3)
                edge node[right]{$\left[
                    \begin{smallmatrix}
                      1 & 0\\
                      0 & 1
                    \end{smallmatrix}
                  \right]$} (d-2-2)
        (d-1-2.south west) edge node[yshift=-3pt, above left]{$M_2$}
          (d-2-1.north east)
        (d-1-3) edge node[above]{$\left[
            \begin{smallmatrix}
              - \alpha_n & 0\\
              \hfill \varphi_n & \alpha_{n-1}
            \end{smallmatrix}
          \right]$} (d-1-4)
        (d-1-4) edge node[above]{$\left[
            \begin{smallmatrix}
              - \Sigma \alpha_1 & 0\\
              \hfill \Sigma \varphi_1 & \alpha_n
            \end{smallmatrix}
          \right]$} (d-1-5)
                edge node[right]{$\left[
                    \begin{smallmatrix}
                      1 & 0\\
                      0 & 1
                    \end{smallmatrix}
                  \right]$} (d-2-4)
        (d-1-5) edge node[right]{$\left[
            \begin{smallmatrix}
              1 & 0\\
              0 & 1
            \end{smallmatrix}
          \right]$} (d-2-5)
        (d-1-5.south west) edge node[yshift=-3pt,above left]{$M_1$}
          (d-2-4.north east)
        (d-2-1) edge node[xshift=2pt, above]{$\left[
            \begin{smallmatrix}
              - \alpha_2 & 0\\
              \hfill \varphi_2 & \alpha_1
            \end{smallmatrix}
          \right]$} (d-2-2)
        (d-2-2) edge node[above]{$\left[
            \begin{smallmatrix}
              - \alpha_3 & 0\\
              \hfill \varphi_3 & \alpha_2
            \end{smallmatrix}
          \right]$} (d-2-3)
        (d-2-3) edge node[above]{$\left[
            \begin{smallmatrix}
              - \alpha_n & 0\\
              \hfill \varphi_n & \alpha_{n-1}
            \end{smallmatrix}
          \right]$} (d-2-4)
        (d-2-4) edge node[xshift=4pt, above]{$\left[
            \begin{smallmatrix}
              - \Sigma \alpha_1 & 0\\
              \hfill \Sigma \varphi_1 & \alpha_n
            \end{smallmatrix}
          \right]$} (d-2-5);
    \end{tikzpicture}
  \end{center}
  with
  \begin{equation*}
    M_1 = \left[
      \begin{smallmatrix}
        0 & \Sigma \varphi_1^{-1} \\
        \hfill 0 & 0
      \end{smallmatrix}
    \right], \hspace{7mm} M_i = \left[
      \begin{smallmatrix}
        0 & \varphi_i^{-1} \\
        \hfill 0 & 0
      \end{smallmatrix}
    \right] \text{ for } 2 \le i \leq n
  \end{equation*}
  displays a contracting homotopy.          
\end{proof} 

The first part of the lemma shows that when one defines a collection
of $n$-angles in a category, in order to endow it with the structure
of a (pre-)$n$-angulated category, then one must include all the
contractible candidate $n$-angles.

\begin{remark}
  \label{rem:Neeman}
  It is easily seen that a direct sum of contractible
  $n$-$\Sigma$-sequences is again contractible, as is any trivial
  $n$-$\Sigma$-sequence. Consequently, a direct sum of trivial
  $n$-$\Sigma$-sequences is contractible. The converse holds if
  idempotents split in the category: in this case every contractible
  candidate $n$-angle is automatically a direct sum of trivial
  $n$-$\Sigma$-sequences. Namely, if
  \begin{center}
    \begin{tikzpicture}
      \diagram{d}{3em}{3.4em}{
        A_1 & A_2 & A_3 & \cdots & A_n & \Sigma A_1\\
        A_1 & A_2 & A_3 & \cdots & A_n & \Sigma A_1\\
      };

      \path[->,midway,font=\scriptsize]
        (d-1-1) edge node[above]{$\alpha_1$} (d-1-2)
                     edge node[left]{$1$} (d-2-1)
        (d-1-2) edge node[above]{$\alpha_2$} (d-1-3)
                     edge node[above left]{$\Theta_1$} (d-2-1)
                     edge node[left]{$1$} (d-2-2)
        (d-1-3) edge node[above]{$\alpha_3$} (d-1-4)
                     edge node[above left]{$\Theta_2$} (d-2-2)
                     edge node[left]{$1$} (d-2-3)
        (d-1-4) edge node[above]{$\alpha_{n - 1}$} (d-1-5)
        (d-1-5) edge node[above]{$\alpha_n$} (d-1-6)
                     edge node[left]{$1$} (d-2-5)
        (d-1-6) edge node[above left]{$\Theta_n$} (d-2-5)
                     edge node[left]{$1$} (d-2-6)
        (d-2-1) edge node[above]{$\alpha_1$} (d-2-2)
        (d-2-2) edge node[above]{$\alpha_2$} (d-2-3)
        (d-2-3) edge node[above]{$\alpha_3$} (d-2-4)
        (d-2-4) edge node[above]{$\alpha_{n - 1}$} (d-2-5)
        (d-2-5) edge node[above]{$\alpha_n$} (d-2-6);
    \end{tikzpicture}
  \end{center}
  is a contracting homotopy on a candidate $n$-angle $A_\bullet$, then
  the equalities
  \begin{equation*}
    \begin{array}{r@{\ }c@{\ }l}
      1_{A_2} & = & (\Theta_2 \circ \alpha_2) + ( \alpha_1 \circ
      \Theta_1)\\
      & \vdots &\\
      1_{A_n} & = & (\Theta_n \circ \alpha_n) + ( \alpha_{n - 1} \circ
      \Theta_{n - 1})\\
      1_{\Sigma A_1} & = & (\Sigma\Theta_1 \circ \Sigma \alpha_1) + (
      \alpha_n \circ \Theta_n)
    \end{array}
  \end{equation*}
  are idempotent decompositions of the identity morphisms on the
  objects. The result is a decomposition of the candidate $n$-angle
  into a sum of trivial $n$-$\Sigma$-sequences. However, if it is not
  the case that all the idempotents split in the category, then a
  contractible candidate $n$-angle need not be a direct sum of trivial
  $n$-$\Sigma$-sequences.
\end{remark}

\section{Higher $n$-angulations}
\label{sec:example}

In this section we prove our main result: for a commutative local ring
with principal maximal ideal squaring to zero, the category of
finitely generated free modules is $n$-angulated for every $n \geq
3$. The construction of the class of $n$-angles is based on that of a
triangulated category without models in \cite{MSS}, and for odd $n$ we
need the same restrictions on the ring as in that paper. However, for
even $n$ our construction is much more general.

We now fix some notation. 

\begin{notation}
  (1) Let $R$ be a commutative local ring with a principal nonzero
  maximal ideal $\m =(p)$ satisfying $\m^2 =0$.

  (2) Let $\C$ be the category of finitely generated free $R$-modules,
  and $\Sigma \colon \C \to \C$ the identity functor.

  (3) Given a free module $F \in \C$, denote by $F(p)_{\bullet}$ the
  $n$-$\Sigma$-sequence
  \begin{equation*}
    F \xrightarrow{p} F \xrightarrow{p} \cdots \xrightarrow{p} F
    \xrightarrow{p} \Sigma F
  \end{equation*}
  in which the maps are just multiplication with the element $p \in R$. 

  (4) Let $\nang$ be the collection of all $n$-$\Sigma$-sequences in
  $\C$ isomorphic to a direct sum of the form $C_{\bullet} \oplus
  F(p)_{\bullet}$, where $C_{\bullet}$ is a contractible candidate
  $n$-angle in $\C$ and $F$ is a free module in $\C$.
\end{notation}

Two standard examples of such a local ring are $\Z/(q^2)$ for a prime
number $q$, and $k[x]/(x+a)^2$ for a field $k$ and any element
$a \in k$. More generally, if $A$ is an integral domain of Krull
dimension one, and $q \in A$ is a prime element, then the ring
$R = A/(q^2)$ satisfies our assumptions.

We now collect some facts about our ring $R$ and its category $\C$ of
finitely generated free modules.

\begin{remark}
  \label{rem:facts}
  (1) Let $a$ be an element of $R$. If $a \notin \m$ then the element
  is a unit. If $a \in \m$ then $a = pb$ for some element $b \in
  R$. Now, either $b$ is a unit, or $b \in \m$, in which case $a=0$
  since $\m^2 =0$. Consequently, an element of $R$ is either zero, a
  unit or of the form $up$ for some unit $u$.

  (2) The ideal $\m$ is the only non-trivial ideal of
  $R$. Consequently, the ring has Krull dimension zero, and the
  finitely generated $R$-modules have finite length. Moreover, by
  applying the Baer Criterion it is easily seen that $R$ is
  selfinjective.

  (3) Suppose $F_1 \xrightarrow{\alpha} F_2$ is an $R$-homomorphism
  between finitely generated free $R$-modules $F_1$ and $F_2$ of ranks
  $t_1$ and $t_2$, say. Then $\alpha$ is given by a $t_2 \times t_1$
  matrix $(a_{ij})$ with entries in $R$. Performing appropriate row
  and column operations on this matrix gives a new matrix of the form
  \begin{equation*}
    \begin{bmatrix}
      pI_u & 0 & 0\\
      0 & I_v & 0\\
      0 & 0 & Z
    \end{bmatrix}
  \end{equation*}
  where $I_u$ and $I_v$ are $u \times u$ and $v \times v$ identity
  matrices, and $Z$ is a zero matrix of size $(t_2-u-v) \times
  (t_1-u-v)$. Consequently, the map $\alpha$ is isomorphic to the
  decomposition
  \begin{equation*}
    F \oplus G \oplus H_1 \xrightarrow{
      \left[
        \begin{smallmatrix}
          p & 0 & 0\\
          0 & 1 & 0\\
          0 & 0 & 0
        \end{smallmatrix}
      \right]}
    F \oplus G \oplus H_2
  \end{equation*}
  for some free $R$-modules $F,G,H_1,H_2$, with $F_1 \simeq F \oplus G
  \oplus H_1$ and $F_2 \simeq F \oplus G \oplus H_2$. More precisely,
  there exists a commutative diagram
  \begin{center}
    \begin{tikzpicture}
      \diagram{d}{2.5em}{3em}{
        F_1 & F_2\\
        F \oplus G \oplus H_1 & F \oplus G \oplus H_2\\
      };

      \path[->,midway,font=\scriptsize]
        (d-1-1) edge node[above]{$\alpha$} (d-1-2)
                     edge node[right]{$\varphi_1$} (d-2-1)
        (d-1-2) edge node[right]{$\varphi_2$} (d-2-2)
        (d-2-1) edge node[above]{$\left[
            \begin{smallmatrix}
              p & 0 & 0\\
              0 & 1 & 0\\
              0 & 0 & 0
            \end{smallmatrix}
          \right]$} (d-2-2);
    \end{tikzpicture}
  \end{center}
  in $\C$, in which the vertical maps $\varphi_1$ and $\varphi_2$ are
  isomorphisms.

  (4) If $F$ is a free $R$-module of rank $t$, say, then the
  $n$-$\Sigma$-sequence $F(p)_{\bullet}$ in $\C$ is just the direct
  sum of $t$ copies of the $n$-$\Sigma$-sequence
  \begin{equation*}
    R \xrightarrow{p} R \xrightarrow{p} \cdots
    \xrightarrow{p} R \xrightarrow{p} \Sigma R 
  \end{equation*}
  i.e.\ $R(p)_{\bullet}$.

  (5) The idempotents in the category $\C$ split, hence the
  contractible candidate $n$-angles in $\C$ are precisely
  the direct sums of trivial $n$-$\Sigma$-sequences (cf.\ Remark
  \ref{rem:Neeman}).

  (6) Contractible candidate $n$-angles in $\C$ are in
  particular exact sequences of free modules, and so are all the
  sequences of the form $F(p)_{\bullet}$. Hence every
  $n$-angle in $\nang$ is exact (and periodic) when viewed as
  a sequence of free modules.

  (7) The collection $\nang$ of $n$-$\Sigma$-sequences in $\C$ depends
  on the integer $n$, so one should perhaps be more precise and write
  $\nang(n)$. However, we will not do this. Also, since the functor
  $\Sigma \colon \C \to \C$ is just the identity functor on $\C$, it
  may seem a bit superfluous to display it in the
  $n$-$\Sigma$-sequences. We have chosen to do so because it makes the
  $n$-$\Sigma$-sequences look more natural.
\end{remark}

We shall now prove that the triple $(\C, \Sigma, \nang)$ is an
$n$-angulated category for every $n \geq 3$. For odd $n$, we need the
additional assumption that $2p=0$ in $R$, but not for even $n$. We
divide the proof into three steps: axiom (N1), axiom (N2) and axiom
(N3)/(N4).

\begin{proposition}[Axiom (N1)]
  \label{prop:N1}
  For every $n \geq 3$, the collection $\nang$ is closed under direct
  sums, direct summands and isomorphisms. Furthermore, it contains all
  the trivial $n$-$\Sigma$-sequences, and every morphism in $\C$ is
  the first morphism of an $n$-$\Sigma$-sequence in $\nang$.
\end{proposition} 

\begin{proof}
  The collection $\nang$ is by definition closed under isomorphisms,
  and it contains all the trivial $n$-$\Sigma$-sequences since these
  are contractible. The direct sum of two contractible
  $n$-$\Sigma$-sequences is again contractible, and if $F,G$ are two
  free modules, then the $n$-$\Sigma$-sequence $F(p)_{\bullet} \oplus
  G(p)_{\bullet}$ equals $(F \oplus G)(p)_{\bullet}$. Hence $\nang$ is
  closed under direct sums.

  Next we show that $\nang$ is closed under direct summands. Let
  \begin{equation*}
    A_{\bullet} \colon \quad A_1 \xrightarrow{\alpha_1} A_2
    \xrightarrow{\alpha_2} \cdots \xrightarrow{\alpha_{n - 1}} A_n
    \xrightarrow{\alpha_n} \Sigma A_1
  \end{equation*}
  be a nonzero direct summand of an $n$-$\Sigma$-sequence in
  $\nang$. If the map $\alpha_1$ is not minimal, i.e.\ if
  $\Im \alpha_1 \not\subseteq \m A_2$, then its matrix must contain a
  unit. Remark \ref{rem:facts}(3) then gives the first square in the
  commutative diagram
  \begin{center}
    \begin{tikzpicture}
      \diagram{d}{2.5em}{3em}{
        A_1 & A_2 & A_3 & \cdots & A_n & \Sigma A_1\\
        R \oplus A'_1 & R \oplus A'_2 & A_3 & \cdots & A_n & \Sigma R
        \oplus \Sigma A'_1\\
      };

      \path[->,midway,font=\scriptsize]
        (d-1-1) edge node[above]{$\alpha_1$} (d-1-2)
                     edge node[right]{$\varphi_1$} (d-2-1)
        (d-1-2) edge node[above]{$\alpha_2$} (d-1-3)
                     edge node[right]{$\varphi_2$} (d-2-2)
        (d-1-3) edge node[above]{$\alpha_3$} (d-1-4)
                     edge node[right]{$1$} (d-2-3)
        (d-1-4) edge node[above]{$\alpha_{n - 1}$} (d-1-5)
        (d-1-5) edge node[above]{$\alpha_n$} (d-1-6)
                     edge node[right]{$1$} (d-2-5)
        (d-1-6) edge node[right]{$\Sigma \varphi_1$} (d-2-6)
        (d-2-1) edge node[above]{$\left[
            \begin{smallmatrix}
              1 & 0\\
              0 & \alpha'_1
            \end{smallmatrix}
          \right]$} (d-2-2)
        (d-2-2) edge node[above]{$\left[
            \begin{smallmatrix}
              \alpha'_2 & \alpha''_2
            \end{smallmatrix}
          \right]$} (d-2-3)
        (d-2-3) edge node[above]{$\alpha_3$} (d-2-4)
        (d-2-4) edge node[above]{$\alpha_{n - 1}$} (d-2-5)
        (d-2-5) edge node[above]{$\left[
            \begin{smallmatrix}
              \alpha'_n\\
              \alpha''_n
            \end{smallmatrix}
          \right]$} (d-2-6);
    \end{tikzpicture}
  \end{center}
  whose vertical maps $\varphi_1$ and $\varphi_2$ are isomorphisms,
  and where
  \begin{equation*}
    \begin{array}{r@{\ }c@{\ }l}
      \left[
        \begin{smallmatrix}
          \alpha'_2 & \alpha''_2 \\
        \end{smallmatrix}
      \right] & = &  \alpha_2 \circ \varphi_2^{-1} \\
      \left[
        \begin{smallmatrix}
          \alpha'_n \\
          \alpha''_n 
        \end{smallmatrix}
      \right] & =  & \Sigma \varphi_1 \circ \alpha_n.
    \end{array}
  \end{equation*}
  This is an isomorphism of $n$-$\Sigma$-sequences, and so since
  $A_{\bullet}$ is a periodic exact sequence of free modules (being a
  summand of one), so is the bottom sequence. Therefore the maps
  $\alpha'_2$ and $\alpha'_n$ are zero, and the trivial
  $n$-$\Sigma$-sequence on $R$, which belongs to $\nang$, splits
  off. Consequently, we may assume that all the maps in $A_{\bullet}$
  are minimal.
  
  Since $A_{\bullet}$ is a complex of free modules, the exact sequence
  \begin{equation*}
    0 \to R/ \m \xrightarrow{p} R \to R/ \m \to 0
  \end{equation*}
  induces an exact sequence
  \begin{equation*}
    0 \to R/ \m \otimes_R A_{\bullet} \to A_{\bullet} \to R/ \m
    \otimes_R A_{\bullet} \to 0
  \end{equation*}
  of complexes. The exactness of $A_{\bullet}$ then gives isomorphisms
  $H_i( R/ \m \otimes_R A_{\bullet} ) \simeq H_{i+1}( R/ \m \otimes_R
  A_{\bullet})$ of homology for all $i$, and so since the maps in
  $A_{\bullet}$ are minimal we see that the free modules $A_1, \dots,
  A_n$ are all of the same rank. Combining this (and the minimality of
  the maps in $A_{\bullet}$) with Remark \ref{rem:facts}(3), we see that
  $A_{\bullet}$ is isomorphic to $F(p)_{\bullet}$ for some free module
  $F$. This shows that $A_\bullet$ belongs to the collection $\nang$.

  It remains to show that every morphism
  $A_1 \xrightarrow{\alpha} A_2$ in $\C$ is the first morphism of an
  $n$-$\Sigma$-sequence in $\nang$. But this is easy: by Remark
  \ref{rem:facts}(3) again, the map $\alpha$ is isomorphic to a map
  \begin{equation*}
      F \oplus G \oplus H_1 \xrightarrow{\left[
          \begin{smallmatrix}
            p & 0 & 0\\
            0 & 1 & 0\\
            0 & 0 & 0
          \end{smallmatrix}
        \right]}
      F \oplus G \oplus H_2
    \end{equation*}
    in $\C$. This map is the first in the $n$-$\Sigma$-sequence which
    is the direct sum of $F(p)_\bullet$ and three trivial
    $n$-$\Sigma$-sequences involving the modules $G, H_1$ and $H_2$.
\end{proof}

Next we show that axiom (N2) holds. Here we need the additional
assumption that $2p = 0$ in $R$ when $n$ is odd; this was not needed for
axiom (N1) to hold. To see why we need this extra assumption for odd
$n$, consider the triangulated case (i.e.\ $n=3$) and the
$3$-$\Sigma$-sequence
\begin{equation*}
  R(p)_\bullet \colon \quad R \xrightarrow{p} R \xrightarrow{p} R
  \xrightarrow{p} \Sigma R
\end{equation*}
in $\nang$. Its left rotation is the $3$-$\Sigma$-sequence
\begin{equation*}
  R \xrightarrow{p} R \xrightarrow{p} \Sigma R \xrightarrow{-p} \Sigma
  R
\end{equation*}
and if this is to belong to $\nang$ then it is easily seen that it
must be isomorphic to $R(p)_\bullet$. If
\begin{center}
  \begin{tikzpicture}
    \diagram{d}{2.5em}{2.5em}{
      R & R & R & \Sigma R\\
      R & R & \Sigma R & \Sigma R\\
    };

    \path[->,midway,font=\scriptsize]
      (d-1-1) edge node[above]{$p$} (d-1-2)
                   edge node[right]{$u$} (d-2-1)
      (d-1-2) edge node[above]{$p$} (d-1-3)
                   edge node[right]{$v$} (d-2-2)
      (d-1-3) edge node[above]{$p$} (d-1-4)
                   edge node[right]{$w$} (d-2-3)
      (d-1-4) edge node[right]{$\Sigma u$} (d-2-4)
      (d-2-1) edge node[above]{$p$} (d-2-2)
      (d-2-2) edge node[above]{$p$} (d-2-3)
      (d-2-3) edge node[above]{$-p$} (d-2-4);
  \end{tikzpicture}
\end{center}
is an isomorphism, then the elements $u,v$ and $w$ are units in $R$,
and commutativity of the three squares (from left to right) gives
\begin{equation*}
  p = u^{-1}pv = u^{-1}pw = u^{-1}(-p) \Sigma u.
\end{equation*}
Since $\Sigma$ is the identity functor, we see that $p = -p$. This
argument can easily be generalized to all odd $n$.

A natural question arises: could the extra assumption be removed if we
just \emph{defined} the collection $\nang$ to contain the left
rotation of the $n$-$\Sigma$-sequence $R(p)_\bullet$ whenever $n$ is
odd? Note that for even $n$ this is irrelevant: in this case
$R(p)_\bullet$ equals its left rotation, since there is no change of
sign. However, for odd $n$ the answer is no: we would still need the
assumption that $2p=0$ in $R$. To see why, consider again the case
$n=3$, and suppose that both $R(p)_\bullet$ and its left rotation
belong to $\nang$. Then by axiom (N3), we can complete the solid part
of the diagram
\begin{center}
  \begin{tikzpicture}
    \diagram{d}{2.5em}{2.5em}{
      R & R & R & \Sigma R\\
      R & R & \Sigma R & \Sigma R\\
    };

    \path[->,midway,font=\scriptsize]
      (d-1-1) edge node[above]{$p$} (d-1-2)
                   edge node[right]{$1$} (d-2-1)
      (d-1-2) edge node[above]{$p$} (d-1-3)
                   edge node[right]{$1$} (d-2-2)
      (d-1-3) edge node[above]{$p$} (d-1-4)
                   edge[densely dashed] node[right]{$w$} (d-2-3)
      (d-1-4) edge node[right]{$1$} (d-2-4)
      (d-2-1) edge node[above]{$p$} (d-2-2)
      (d-2-2) edge node[above]{$p$} (d-2-3)
      (d-2-3) edge node[above]{$-p$} (d-2-4);
        \end{tikzpicture}
\end{center}
to a morphism of $3$-$\Sigma$-sequences. The same argument as above
then forces $p$ and $-p$ to be equal in $R$, and this can also be
generalized to all odd $n$.

\begin{proposition}[Axiom (N2)]
  \label{prop:N2}
  Suppose that $n\geq 3$ is an integer, and that $2p=0$ in $R$
  whenever $n$ is odd. Then an $n$-$\Sigma$-sequence belongs to
  $\nang$ if and only if its left rotation does.
\end{proposition} 

\begin{proof}
  A candidate $n$-angle is obviously contractible if and only if its
  left rotation is, regardless of the extra assumption. For a free
  module $F \in \C$, the left rotation of $F(p)_\bullet$ is the
  $n$-$\Sigma$-sequence
  \begin{equation*}
    F \xrightarrow{p} F \xrightarrow{p} \cdots \xrightarrow{p}
    \Sigma F \xrightarrow{(-1)^np} \Sigma F.
  \end{equation*}
  Since the functor $\Sigma$ is the identity, this
  $n$-$\Sigma$-sequence equals $F(p)_\bullet$ whenever $n$ is even,
  or whenever $n$ is odd and $2p=0$ in $R$.
\end{proof}

Before we show that axioms (N3) and (N4) hold, we show that they hold
if we only consider sequences of the form $F(p)_\bullet$ for a free
module $F$. The proof that the two axioms hold for
\emph{all} sequences in $\nang$ reduces to this special case.

\begin{lemma}
  \label{lem:cone}
  Suppose that $n\geq 3$ is an integer, and that $2p=0$ in $R$
  whenever $n$ is odd. Then for all free modules $F,G \in \C$, the
  solid part of each commutative diagram
  \begin{center}
    \begin{tikzpicture}
      \diagram{d}{2.5em}{2.5em}{
        F & F & F & \cdots & F & \Sigma F \\
        G & G & G & \cdots & G & \Sigma G \\
      };

      \path[->,midway,font=\scriptsize]
        (d-1-1) edge node[above]{$p$} (d-1-2)
                     edge node[right]{$\psi_1$} (d-2-1)
        (d-1-2) edge node[above]{$p$} (d-1-3)
                     edge node[right]{$\psi_2$} (d-2-2)
        (d-1-3) edge node[above]{$p$} (d-1-4)
                     edge[densely dashed] node[right]{$\psi_3$} (d-2-3)
        (d-1-4) edge node[above]{$p$} (d-1-5)
        (d-1-5) edge node[above]{$p$} (d-1-6)
                     edge[densely dashed] node[right]{$\psi_n$} (d-2-5)
        (d-1-6) edge node[right]{$\Sigma \psi_{1}$} (d-2-6)
        (d-2-1) edge node[above]{$p$} (d-2-2)
        (d-2-2) edge node[above]{$p$} (d-2-3)
        (d-2-3) edge node[above]{$p$} (d-2-4)
        (d-2-4) edge node[above]{$p$} (d-2-5)
        (d-2-5) edge node[above]{$p$} (d-2-6);
    \end{tikzpicture}
  \end{center}
  can be completed to a morphism of $n$-$\Sigma$-sequences, in such a
  way that the mapping cone belongs to $\nang$.
\end{lemma} 

\begin{proof}
  Decompose the matrices $\psi_1$ and $\psi_2$ as
  $\psi_i = \psi_i' + p \theta_i$, where $\psi_i'$ is a matrix whose
  entries are just units (if any) and zeros. The equality
  $p \psi_1 = p \psi_2$ implies that $p \psi_1' = p \psi_2'$, hence
  the matrices $\psi_1'$ and $\psi_2'$ must be equal; we denote this
  matrix by $\psi$. Since $p \psi_i = p \psi$, we can complete the
  given diagram to a morphism of $n$-$\Sigma$-sequences by taking
  $\psi_3 = \psi_2-p \theta_1$ and $\psi_4 = \cdots = \psi_n = \psi$
  (if $n \geq 4$), giving the morphism $(\psi_1, \psi_2, \psi_2-p
  \theta_1, \psi, \dots, \psi)$. We shall prove that the mapping cone
  of this morphism belongs to $\nang$.

  Consider the even simpler morphism $(\psi, \psi, \dots, \psi)$
  between $F(p)_\bullet$ and $G(p)_\bullet$. The diagram
  \begin{center}
    \begin{tikzpicture}
      \diagram{d}{4em}{4.5em}{
        F & F & F & F & \cdots & F & \Sigma F \\
        G & G & G & G & \cdots & G & \Sigma G \\
      };

      \path[->,midway,font=\scriptsize]
        (d-1-1) edge node[above]{$p$} (d-1-2)
        ([xshift=-1mm] d-1-1.south) edge node[left]{$\psi_1$}
        ([xshift=-1mm] d-2-1.north)
        ([xshift=1mm] d-1-1.south) edge node[right]{$\psi$}
        ([xshift=1mm] d-2-1.north)
        (d-1-2) edge node[above]{$p$} (d-1-3)
                     edge node[yshift=-2pt,above left,
                       pos=0.3]{$\theta_1$} (d-2-1)
        ([xshift=-1mm] d-1-2.south) edge node[left]{$\psi_2$}
        ([xshift=-1mm] d-2-2.north)
        ([xshift=1mm] d-1-2.south) edge node[right]{$\psi$}
        ([xshift=1mm] d-2-2.north)
        (d-1-3) edge node[above]{$p$} (d-1-4)
                     edge node[yshift=-2pt,above left,
                       pos=0.3]{$\theta_2 - \theta_1$} (d-2-2)
        ([xshift=-1mm] d-1-3.south) edge
          node[left]{$\overline{\psi}_2$}
        ([xshift=-1mm] d-2-3.north)
        ([xshift=1mm] d-1-3.south) edge node[right]{$\psi$}
        ([xshift=1mm] d-2-3.north)
        (d-1-4) edge node[above]{$p$} (d-1-5)
                     edge node[yshift=-2pt,above left, pos=0.3]{$0$}
                       (d-2-3)
        ([xshift=-1mm] d-1-4.south) edge node[left]{$\psi$}
        ([xshift=-1mm] d-2-4.north)
        ([xshift=1mm] d-1-4.south) edge node[right]{$\psi$}
        ([xshift=1mm] d-2-4.north)
        (d-1-5) edge node[above]{$p$} (d-1-6)
        ([xshift=-1mm] d-1-6.south) edge node[left]{$ \psi$}
        ([xshift=-1mm] d-2-6.north)
        ([xshift=1mm] d-1-6.south) edge node[right]{$\psi$}
        ([xshift=1mm] d-2-6.north)
        (d-1-6) edge node[above]{$p$} (d-1-7)
        (d-1-7) edge node[yshift=-2pt,above left, pos=0.3]{$0$}
          (d-2-6)
        ([xshift=-1mm] d-1-7.south) edge node[left]{$\Sigma \psi_1$}
        ([xshift=-1mm] d-2-7.north)
        ([xshift=1mm] d-1-7.south) edge node[right]{$\Sigma \psi$}
        ([xshift=1mm] d-2-7.north)
        (d-2-1) edge node[above]{$p$} (d-2-2)
        (d-2-2) edge node[above]{$p$} (d-2-3)
        (d-2-3) edge node[above]{$p$} (d-2-4)
        (d-2-4) edge node[above]{$p$} (d-2-5)
        (d-2-5) edge node[above]{$p$} (d-2-6)
        (d-2-6) edge node[above]{$p$} (d-2-7);
    \end{tikzpicture}
  \end{center}
  displays a homotopy between the two morphisms (where
  $\overline{\psi}_2 = \psi_2 -p \theta_1$), hence by Lemma
  \ref{lem:homotopic}(2) the mapping cone of $( \psi_1, \psi_2,
  \psi_2-p \theta_1, \psi, \dots, \psi )$ is isomorphic to that of $(
  \psi, \psi, \dots, \psi )$. It suffices therefore to show that the
  mapping cone of $( \psi, \psi, \dots, \psi )$ belongs to $\nang$.

  If $\psi = 0$, then the mapping cone is just the direct sum of
  $G(p)_\bullet$ and $\overline{F(p)}_\bullet$, where the latter is
  $F(p)_\bullet$ but with a sign change on all the maps. When $n$ is
  odd, the assumption $2p=0$ gives
  $\overline{F(p)}_\bullet = F(p)_\bullet$, whereas when $n$ is even,
  the diagram
  \begin{center}
    \begin{tikzpicture}
      \diagram{d}{2.5em}{2.5em}{
        F & F & F & \cdots & F & \Sigma F \\
        F & F & F & \cdots & F & \Sigma F \\
      };

      \path[->,midway,font=\scriptsize]
        (d-1-1) edge node[above]{$p$} (d-1-2)
                     edge node[right]{$1$} (d-2-1)
        (d-1-2) edge node[above]{$p$} (d-1-3)
                     edge node[right]{$-1$} (d-2-2)
        (d-1-3) edge node[above]{$p$} (d-1-4)
                     edge node[right]{$1$} (d-2-3)
        (d-1-4) edge node[above]{$p$} (d-1-5)
        (d-1-5) edge node[above]{$p$} (d-1-6)
                     edge node[right]{$-1$} (d-2-5)
        (d-1-6) edge node[right]{$1$} (d-2-6)
        (d-2-1) edge node[above]{$-p$} (d-2-2)
        (d-2-2) edge node[above]{$-p$} (d-2-3)
        (d-2-3) edge node[above]{$-p$} (d-2-4)
        (d-2-4) edge node[above]{$-p$} (d-2-5)
        (d-2-5) edge node[above]{$-p$} (d-2-6);
    \end{tikzpicture}
  \end{center}
  shows that the $n$-$\Sigma$-sequences $F(p)_\bullet$ and
  $\overline{F(p)}_\bullet$ are isomorphic. In either case the
  sequence $\overline{F(p)}_\bullet$, and therefore also the mapping
  cone of the morphism $(0,0, \dots, 0)$, belongs to $\nang$.

  Suppose $\psi$ is nonzero, so that it contains at least one unit. As
  in Remark \ref{rem:facts}(3), by performing appropriate row and
  column operations on $\psi$, we obtain a matrix $\widetilde{\psi}$
  of the form
  \begin{equation*}
    \begin{bmatrix}
      I & 0\\
      0 & Z
    \end{bmatrix}
  \end{equation*}
  where $I$ is a square identity matrix and $Z$ a zero matrix. The
  mapping cone of $(\psi, \dots, \psi)$ is then isomorphic to that of
  $(\widetilde{\psi}, \dots, \widetilde{\psi})$, hence it suffices to
  show that the latter belongs to $\nang$.

  The form xsof the matrix $\widetilde{\psi}$ implies that the modules
  $F$ and $G$ decompose in $\C$ as $F = F_1 \oplus F_2$ and
  $G = F_1 \oplus G_2$, such that the map
  \begin{equation*}
    F_1 \oplus F_2 \xrightarrow{ \left[
        \begin{smallmatrix}
          1 & 0\\
          0 & 0
        \end{smallmatrix}
      \right] } F_1 \oplus G_2
  \end{equation*}
  equals the map $\widetilde{\psi}$. The mapping cone of the morphism
  $(\widetilde{\psi}, \dots, \widetilde{\psi} )$ is then the direct
  sum of three $n$-$\Sigma$-sequences, two of which are
  $G_2(p)_\bullet$ and $\overline{F_2(p)}_\bullet$, where we have used
  the same notation as earlier in the proof. As we have seen, the
  sequence $\overline{F_2(p)}_\bullet$ (and also $G_2(p)_\bullet$)
  belong to $\nang$. The third summand is the mapping cone of the
  identity morphism on $F_1(p)_{\bullet}$, and by Lemma
  \ref{lem:contractible}(2), this $n$-$\Sigma$-sequence is
  contractible and belongs to $\nang$. This shows that the mapping
  cone of $(\widetilde{\psi}, \dots, \widetilde{\psi} )$ belongs to
  $\nang$.
\end{proof}

Having proved the special case, we now show that axioms (N3) and (N4)
hold for all sequences in $\nang$.

\begin{proposition}[Axiom (N3)/(N4)]
  \label{prop:N3/N4}
  Suppose that $n\geq 3$ is an integer, and that $2p=0$ in $R$
  whenever $n$ is odd. Then the solid part of each commutative diagram
  \begin{center}
    \begin{tikzpicture}
      \diagram{d}{2.5em}{2.5em}{
        A_1 & A_2 & A_3 & \cdots & A_n & \Sigma A_1\\
        B_1 & B_2 & B_3 & \cdots & B_n & \Sigma B_1\\
      };

      \path[->,midway,font=\scriptsize]
        (d-1-1) edge node[above]{$\alpha_1$} (d-1-2)
                     edge node[right]{$\varphi_1$} (d-2-1)
        (d-1-2) edge node[above]{$\alpha_2$} (d-1-3)
                     edge node[right]{$\varphi_2$} (d-2-2)
        (d-1-3) edge node[above]{$\alpha_3$} (d-1-4)
                     edge[densely dashed] node[right]{$\varphi_3$} (d-2-3)
        (d-1-4) edge node[above]{$\alpha_{n - 1}$} (d-1-5)
        (d-1-5) edge node[above]{$\alpha_n$} (d-1-6)
                     edge[densely dashed] node[right]{$\varphi_n$} (d-2-5)
        (d-1-6) edge node[right]{$\Sigma\varphi_{1}$} (d-2-6)
        (d-2-1) edge node[above]{$\beta_1$} (d-2-2)
        (d-2-2) edge node[above]{$\beta_2$} (d-2-3)
        (d-2-3) edge node[above]{$\beta_3$} (d-2-4)
        (d-2-4) edge node[above]{$\beta_{n - 1}$} (d-2-5)
        (d-2-5) edge node[above]{$\beta_n$} (d-2-6);
    \end{tikzpicture}
  \end{center}
  with rows in $\nang$ can be completed to a morphism of
  $n$-$\Sigma$-sequences, in such a way that the mapping cone belongs
  to $\nang$.
\end{proposition} 

\begin{proof}
  Suppose that $A_\bullet$ contains a trivial $n$-$\Sigma$-sequence
  $T_\bullet$ as a direct summand, so that $A_\bullet$ is isomorphic
  to a direct sum $T_\bullet \oplus A'_\bullet$ for some
  $n$-$\Sigma$-sequence $A'_\bullet \in \nang$ (both $T_\bullet$ and
  $A'_\bullet$ belong to $\nang$ by Proposition \ref{prop:N1}). Then
  the maps $\varphi_1$ and $\varphi_2$ decompose as
  $\varphi_i = \left[ \begin{smallmatrix} \varphi_i^T &
      \varphi_i' \end{smallmatrix} \right]$,
  where $\varphi_i^T \in \Hom_{\C}(T_i,B_i)$ and
  $\varphi'_i \in \Hom_{\C}(A'_i,B_i)$. Since $T_\bullet$ is a trivial
  $n$-$\Sigma$-sequence, it is easily seen that the two maps
  $\varphi_1^T$ and $\varphi_2^T$ can be completed to a morphism
  $\varphi^T \colon T_\bullet \to B_\bullet$ of
  $n$-$\Sigma$-sequences.
  
  Suppose now that the two maps $\varphi'_1$ and $\varphi'_2$ can be
  completed to a morphism $\varphi' \colon A'_\bullet \to B_\bullet$
  of $n$-$\Sigma$-sequences in such a way that the mapping cone
  $C_{\varphi'}$ belongs to $\nang$. Together, the two morphisms
  $\varphi^T$ and $\varphi'$ form a morphism
  \begin{equation*}
    \left[
      \begin{smallmatrix}
        \varphi^T & \varphi'
      \end{smallmatrix}
    \right] \colon T_\bullet \oplus A'_\bullet \to B_\bullet
  \end{equation*}
  whose first two vertical maps are $\varphi_1$ and $\varphi_2$. Now
  since $T_\bullet$ is trivial (and therefore contractible), the
  morphism $\varphi^T$ is nullhomotopic by Lemma
  \ref{lem:homotopic}(1). But then $\left[ \begin{smallmatrix}
      \varphi^T & \varphi' \end{smallmatrix} \right]$ is homotopic to
  the morphism $\left[ \begin{smallmatrix} 0 &
      \varphi' \end{smallmatrix} \right]$, and so by Lemma
  \ref{lem:homotopic}(2) these two morphisms have isomorphic mapping
  cones. The mapping cone of $\left[ \begin{smallmatrix} 0 &
      \varphi' \end{smallmatrix} \right]$ is the direct sum of
  $C_{\varphi'}$ and the left rotation of $T_\bullet$, the latter
  possibly with a sign change in the only nonzero map. This is easily
  seen to be isomorphic to the left rotation of $T_\bullet$, hence the
  mapping cone of $\left[ \begin{smallmatrix} 0 &
      \varphi' \end{smallmatrix} \right]$, and therefore also the
  mapping cone of $\left[ \begin{smallmatrix} \varphi^T &
      \varphi' \end{smallmatrix} \right]$, belongs to $\nang$.

  The above shows that we can ``remove'' any trivial summands of
  $A_\bullet$. Similarly, we can remove trivial summands of
  $B_\bullet$; a similar argument holds, or we can use the above
  argument together with the duality $\Hom_R(-,R)$ on $\C$ (the ring
  $R$ is selfinjective, cf.\ Remark \ref{rem:facts}(2)). Since the
  contractible candidate $n$-angles are just direct sums of trivial
  ones (cf.\ Remark \ref{rem:facts}(5)), it follows that we can reduce
  to the case when $A_\bullet = F(p)_\bullet$ and
  $B_\bullet = G(p)_\bullet$ for free modules $F$ and $G$. Then the
  result is just the previous one, Lemma \ref{lem:cone}.
\end{proof}

We summarize everything in the following theorem. 

\begin{theorem}
  \label{thm:main}
  Let $R$ be a commutative local ring with maximal principal ideal $\m
  = (p)$ satisfying $\m^2 =0$, and $\C$ the category of finitely
  generated free $R$-modules. Furthermore, let $n \geq 3$ be an
  integer and $\Sigma \colon \C \to \C$ the identity functor. Finally,
  let $\nang$ be the collection of all $n$-$\Sigma$-sequences
  isomorphic to a direct sum $C_\bullet \oplus F(p)_\bullet$,
  where $C_\bullet$ is a contractible candidate $n$-angle, $F$ is a
  free module in $\C$ and $F(p)_\bullet$ is the
  $n$-$\Sigma$-sequence
  \begin{equation*}
    F \xrightarrow{p} F \xrightarrow{p} \cdots \xrightarrow{p} F
    \xrightarrow{p} \Sigma F.
  \end{equation*}
  Then $(\C, \Sigma, \nang)$ is an $n$-angulated category whenever $n$
  is even, or when $n$ is odd and $2p =0$ in $R$.
\end{theorem}

\section{Classes of $n$-angles}
\label{sec:properties}

In this section we explore some properties of the $n$-angulated
categories we have constructed. The question we address is: how
many collections of $n$-angles does the underlying suspended category
$(\C, \Sigma)$ admit? For a general triangulated category, this was
studied in \cite{Balmer}, where an example of an algebraic suspended
category with infinitely many triangulations was given. As Balmer
notes in \emph{loc.\ cit.}, for the (topological) stable homotopy
category this is implicit in \cite{Heller}.

We have already seen a glimpse of what to come. Recall from the
discussion preceding Proposition \ref{prop:N2} that for odd $n$, the
$n$-$\Sigma$-sequence
\begin{equation*}
  R \xrightarrow{p} R \xrightarrow{p} \cdots \xrightarrow{p} R
  \xrightarrow{-p} \Sigma R
\end{equation*}
cannot belong to the collection of $n$-angles we have defined so far,
unless we require that $p = -p$. On the other hand, it must belong to
the collection if the rotation axiom (N2) is to be satisfied, and this
is precisely why we must require that $p = -p$ for odd $n$. It is
straightforward to show that this $n$-$\Sigma$-sequence is isomorphic
to the sequence
\begin{equation*}
  R \xrightarrow{-p} R \xrightarrow{p} \cdots \xrightarrow{p} R
  \xrightarrow{p} \Sigma R
\end{equation*}
and that also for even $n$ it cannot belong to the collection of
$n$-angles unless $p = -p$. But what if, for even $n$, we define the
collection of $n$-angles to contain this new sequence \emph{instead}
of the sequence
\begin{equation*}
  R \xrightarrow{p} R \xrightarrow{p} \cdots \xrightarrow{p} R
  \xrightarrow{p} \Sigma R
\end{equation*}
that we have used so far? The following result shows that we would
still get an $n$-angulated category. In fact, for every unit $u$ in
$R$, the $n$-$\Sigma$-sequence
\begin{equation*}
  R \xrightarrow{up} R \xrightarrow{p} \cdots \xrightarrow{p} R
  \xrightarrow{p} \Sigma R
\end{equation*}
gives rise to an $n$-angulation of $(\C, \Sigma)$, and every
$n$-angulation is obtained this way. Moreover, two such
$n$-angulations coincide if and only if the defining units $u$ and $v$
satisfy $up = vp$ in the ring $R$. These various $n$-angulations arise
from global automorphisms on the underlying category, introduced by
Balmer in \cite{Balmer}. One can obtain the result by applying
\cite[Proposition 3.4]{GKO}.

\begin{theorem}
  \label{thm:several}
  Let $R$ be a commutative local ring with maximal principal ideal
  $\m = (p)$ satisfying $\m^2 =0$, and $\C$ the category of finitely
  generated free $R$-modules. Furthermore, let $n \geq 3$ be an
  integer and $\Sigma \colon \C \to \C$ the identity functor. Finally,
  for each unit $u \in R$, let $\nang_u$ be the collection of all
  $n$-$\Sigma$-sequences isomorphic to a direct sum
  $C_\bullet \oplus F(p)_\bullet$, where $C_\bullet$ is a contractible
  candidate $n$-angle, $F$ is a free module in $\C$ and $F(p)_\bullet$
  is the $n$-$\Sigma$-sequence
  \begin{equation*}
    F \xrightarrow{up} F \xrightarrow{p} \cdots
    \xrightarrow{p} F \xrightarrow{p} \Sigma F 
  \end{equation*}
  Then the following hold:

  (1) $(\C,\Sigma,\nang_u)$ is an $n$-angulated category whenever $n$
  is even, or when $n$ is odd and $2p =0$ in $R$.

  (2) If $\nang$ is any $n$-angulation of $(\C, \Sigma)$, then
  $\nang = \nang_u$ for some unit $u$ in $R$.

  (3) $\nang_u = \nang_v$ if and only if $up=vp$ in $R$.
\end{theorem}

\begin{proof}
  (1) The proofs of Proposition \ref{prop:N1}, Proposition
  \ref{prop:N2}, Lemma \ref{lem:cone} and Proposition \ref{prop:N3/N4}
  all carry over almost verbatim. Only the proof of the rotation axiom
  (N2) needs some clarification.

  Rotating (left) the sequence
  \begin{equation*}
    F \xrightarrow{up} F \xrightarrow{p} \cdots
    \xrightarrow{p} F \xrightarrow{p} \Sigma F 
  \end{equation*}
  once gives the sequence 
  \begin{equation*}
    F \xrightarrow{p} F \xrightarrow{p} \cdots
    \xrightarrow{p} F \xrightarrow{up} \Sigma F 
  \end{equation*}
  that is, the last map (map $n$) becomes multiplication by $up$
  instead of the first one. By rotating twice, map $n - 1$ becomes
  multiplication by $up$, and so on. The point is that all these
  rotations are isomorphic to the original sequence: the diagram
  \begin{center}
    \begin{tikzpicture}
      \diagram{d}{2.5em}{2.5em}{
        F & F & \cdots & F & \Sigma F\\
        F & F & \cdots & F & \Sigma F\\
      };

      \path[->,midway,font=\scriptsize]
        (d-1-1) edge node[above]{$up$} (d-1-2)
                     edge node[right]{$u$} (d-2-1)
        (d-1-2) edge node[above]{$p$} (d-1-3)
                     edge node[right]{$1$} (d-2-2)
        (d-1-3) edge node[above]{$p$} (d-1-4)
        (d-1-4) edge node[above]{$p$} (d-1-5)
                     edge node[right]{$1$} (d-2-4)
        (d-1-5) edge node[right]{$u$} (d-2-5)
        (d-2-1) edge node[above]{$p$} (d-2-2)
        (d-2-2) edge node[above]{$p$} (d-2-3)
        (d-2-3) edge node[above]{$p$} (d-2-4)
        (d-2-4) edge node[above]{$up$} (d-2-5);
    \end{tikzpicture}
  \end{center}
  displays an isomorphism after rotating once.

  (2) Let $\nang$ be an $n$-angulation of $(\C, \Sigma)$. By axiom
  (N1)(c), the map $R \xrightarrow{p} R$ in $\C$ is the first map of
  some $n$-angle
  \begin{equation*}
    A_\bullet \colon \quad R \xrightarrow{p} R
    \xrightarrow{\alpha_2} A_3 \xrightarrow{\alpha_3} \cdots
    \xrightarrow{\alpha_{n - 1}} A_n \xrightarrow{\alpha_n} \Sigma R 
  \end{equation*}
  in $\nang$. As in the proof of Proposition \ref{prop:N1}, we may
  assume that all the maps are minimal (i.e.\ that $\Im \alpha_i
  \subseteq \m A_{i+1}$): otherwise we can split off trivial
  $n$-$\Sigma$-sequences. Now consider this $n$-angle as a periodic
  exact sequence of free $R$-modules. As such, it defines a minimal
  free resolution
  \begin{equation*}
    \cdots \xrightarrow{\alpha_n} R \xrightarrow{p} R
    \xrightarrow{\alpha_2} A_3 \xrightarrow{\alpha_3} \cdots
    \xrightarrow{\alpha_{n - 1}} A_n \xrightarrow{\alpha_n} R \to k \to
    0
  \end{equation*}
  of the residue field $k = R/\m$ of $R$, since $k$ is just isomorphic
  to the image of the map $R \xrightarrow{p} R$. Consequently, the
  residue field $k$ has a minimal free resolution in which the ranks
  of the free modules are bounded. By \cite{Avramov, Eisenbud,
    Gulliksen}, this can only happen for local rings which are
  complete intersections of codimension one, that is, hypersurface
  rings. Moreover, the minimal free resolution of $k$ eventually
  becomes two-periodic with constant rank. In our situation the
  resolution is periodic from the start, and since it contains two
  consecutive terms of rank one, we see that \emph{all} the free
  modules $A_3, \dots, A_n$ must have rank one. Therefore, the
  $n$-angle $A_\bullet$ is of the form
  \begin{equation*}
    R \xrightarrow{p} R \xrightarrow{u_2p} R \xrightarrow{u_3p} \cdots
    \xrightarrow{u_{n - 1}p} R \xrightarrow{u_np} \Sigma R 
  \end{equation*}
  for some units $u_2, \dots, u_n$ in $R$ (the maps cannot be zero or
  isomorphisms, since this would contradict exactness and
  minimality). The diagram
  \begin{center}
    \begin{tikzpicture}
      \diagram{d}{3em}{4em}{
        R & R & R & \cdots & R & \Sigma R \\
        R & R & R & \cdots & R & \Sigma R \\
      };
      
      \path[->,midway,font=\scriptsize]
        (d-1-1) edge node[above]{$p$} (d-1-2)
                     edge node[right]{$1$} (d-2-1)
        (d-1-2) edge node[above]{$u_2p$} (d-1-3)
                     edge node[right]{$\prod_{i=2}^n u_i$} (d-2-2)
        (d-1-3) edge node[above]{$u_3p$} (d-1-4)
                     edge node[right]{$\prod_{i=3}^n u_i$} (d-2-3)
        (d-1-4) edge node[above]{$u_{n - 1}p$} (d-1-5)
        (d-1-5) edge node[above]{$u_np$} (d-1-6)
                     edge node[right]{$u_n$} (d-2-5)
        (d-1-6) edge node[right]{$1$} (d-2-6)
        (d-2-1) edge node[above]{$\left( \prod_{i=2}^n u_i \right) p$}
          (d-2-2)
        (d-2-2) edge node[above]{$p$} (d-2-3)
        (d-2-3) edge node[above]{$p$} (d-2-4)
        (d-2-4) edge node[above]{$p$} (d-2-5)
        (d-2-5) edge node[above]{$p$} (d-2-6);
    \end{tikzpicture}
  \end{center}
  shows that the $n$-angle is isomorphic to the $n$-$\Sigma$-sequence
  \begin{equation*}
    R \xrightarrow{up} R \xrightarrow{p} \cdots \xrightarrow{p} R
    \xrightarrow{p} \Sigma R
  \end{equation*}
  with $u = \prod_{i=2}^n u_i$, hence by axiom (N1)(a) the latter must
  also be an $n$-angle in $\nang$.

  Since $\nang$ is closed under direct sums, it must contain all
  $n$-$\Sigma$-sequences of the form
  \begin{equation*}
    F \xrightarrow{up} F \xrightarrow{p} \cdots \xrightarrow{p} F
    \xrightarrow{p} \Sigma F
  \end{equation*}
  where $F$ is a free module in $\C$: this sequence is just the direct
  sum of copies of the above $n$-angle. Moreover, we know from Lemma
  \ref{lem:contractible} that $\nang$ contains all the contractible
  candidate $n$-angles. Using axiom (N1)(a) again, we see that the
  collection $\nang$ must contain the collection $\nang_u$, and so
  $\nang = \nang_u$ by \cite[Proposition 1.5(c)]{GKO}.

  (3) Let $u$ and $v$ be units in $R$. If $up = vp$ then the
  collections $\nang_u$ and $\nang_v$ are equal by
  definition. Conversely, suppose that $\nang_u$ equals $\nang_v$. By
  axiom (N3), the solid part of the commutative diagram
  \begin{center}
    \begin{tikzpicture}
      \diagram{d}{2em}{2.5em}{
        R & R & R & \cdots & R & \Sigma R \\
        R & R & R & \cdots & R & \Sigma R \\
      };

      \path[->,midway,font=\scriptsize]
        (d-1-1) edge node[above]{$up$} (d-1-2)
                     edge node[right]{$u$} (d-2-1)
        (d-1-2) edge node[above]{$p$} (d-1-3)
                     edge node[right]{$v$} (d-2-2)
        (d-1-3) edge node[above]{$p$} (d-1-4)
                     edge[densely dashed] node[right]{$w_3$} (d-2-3)
        (d-1-4) edge node[above]{$p$} (d-1-5)
        (d-1-5) edge node[above]{$p$} (d-1-6)
                     edge[densely dashed] node[right]{$w_n$} (d-2-5)
        (d-1-6) edge node[right]{$u$} (d-2-6)
        (d-2-1) edge node[above]{$vp$} (d-2-2)
        (d-2-2) edge node[above]{$p$} (d-2-3)
        (d-2-3) edge node[above]{$p$} (d-2-4)
        (d-2-4) edge node[above]{$p$} (d-2-5)
        (d-2-5) edge node[above]{$p$} (d-2-6);
    \end{tikzpicture}
  \end{center} 
  can be completed to a morphism of $n$-angles. Commutativity of the
  squares gives (from left to right, say)
  \begin{equation*}
    u(vp) = u(w_3p) = \cdots = u(w_np) = u(up),
  \end{equation*}
  hence $vp=up$ in $R$.
\end{proof}

\begin{corollary}
  \label{cor:equivalence}
  With the notation from \emph{Theorem \ref{thm:several}}, define an
  equivalence relation on the set of units in $R$ by $u \sim v$ if and
  only if $up = vp$. Then the assignment $[u] \mapsto \nang_u$ is a
  bijective correspondence between the set of equivalence classes and
  the set of $n$-angulations of the suspended category $(\C,\Sigma)$.
\end{corollary}

Depending on the ring $R$, the number of $n$-angulations of $(\C,
\Sigma)$ can vary. We include some examples showing that there could
be infinitely many, finitely many or just one.

\begin{example}
  (1) Let $R$ be the ring $\Z/(4)$. This ring has two units,
  namely $1$ and $3$, and they are equivalent since $1 \cdot 2 = 3
  \cdot 2$ in $R$. Thus for every $n \geq 3$, the suspended category
  $(\C, \Sigma)$ admits only one $n$-angulation. The triangulated
  case, that is, the case $n=3$, is \cite[Theorem 1]{MSS}.

  More generally, let $p$ be a prime number and $R$ the ring
  $\Z/(p^2)$. Then $2p \neq 0$ in $R$ unless $p=2$, so when $p \neq 2$
  we can only consider $n$-angulations on $(\C, \Sigma)$ for even
  $n$. The number of units in $R$ is given by the Euler
  $\phi$-function: it is $\phi (p^2) = p(p-1)$, and they are (the
  congruence classes modulo $p^2$) of the integers
  \begin{center}
    \begin{tabular}{l}
      $1, \dots, p-1,$\\
      $p+1, \dots, 2p-1,$\\[2pt]
      \hspace{2mm} \vdots\\[2pt]
      $p^2-p+1, \dots, p^2-1.$
    \end{tabular}
  \end{center}
  Two units $u,v$ are congruent in $R$ (i.e.\ $up=vp$) if and only if
  they are congruent in $\Z$ modulo $p$, and the above list contains
  $p-1$ such congruence classes, namely $[1], \dots, [p-1]$.
  Consequently, for even $n$ the suspended category $(\C, \Sigma)$
  admits precisely $p-1$ different $n$-angulations.

  (2) Let $k$ be a field and $R$ the ring $k[x]/(x^2)$. If $k$ is
  infinite then $R$ contains infinitely many units, and all such units
  of the form $u$ for some $u \in k$ are incongruent. That is, if $u$
  and $v$ are nonzero different elements in $k$, then considered as
  units in $R$ they satisfy $ux \neq vx$. If in addition the
  characteristic of $k$ is two, then $2x=0$ in $R$, hence our
  construction gives $n$-angulations of the suspended category $(\C,
  \Sigma)$ for all $n \geq 3$ (we do not have to restrict to even
  $n$). To sum up: when $k$ is infinite and of characteristic two,
  then for all $n \geq 3$ the suspended category $(\C, \Sigma)$ admits
  infinitely many different $n$-angulations.
\end{example}

%
%
\section{Algebraic $n$-angulated categories}
\label{sec:algebraic}

Algebraic $n$-angulated categories where introduced in \cite{J}. These
are, as their name suggests, higher analogs of algebraic triangulated
categories.  The aim of this section is to show that, for odd $n$,
some of the $n$-angulated categories constructed in Theorem
\ref{thm:main} are not algebraic.

We recall some definitions and results from \emph{loc.\ cit.}. Let
$\E$ be an additive category. A complex
\begin{equation*}
  A_1 \to A_2 \to \cdots \to A_{n - 1} \to A_n\to 0
\end{equation*}
in $\E$ is called a \emph{right $(n - 2)$-exact sequence\footnote{We
    borrow this terminology from \cite{L}.}} if for every $A\in\E$ the
induced sequence of abelian groups
\begin{equation*}
  0 \to \Hom_\E(A_n,A) \to \Hom_\E(A_{n - 1},A) \to \cdots \to
  \Hom_\E(A_2,A) \to \Hom_\E(A_1,A)
\end{equation*}
is exact. We define \emph{left $(n - 2)$-exact sequences} dually. An
\emph{$(n - 2)$-exact sequence} is a complex which is both a right
$(n - 2)$-exact sequence and a left $(n - 2)$-exact sequence. An
$(n - 2)$-exact sequence is \emph{contractible} if it is contractible
as a complex.

A morphism of complexes in $\E$ of the form
\begin{center}
  \begin{tikzpicture}
    \diagram{d}{2.5em}{2.5em}{
      A_1 & A_2 & \cdots & A_{n - 1}\\
      B_1 & B_2 & \cdots & B_{n - 1}\\
    };

    \path[->,midway,font=\scriptsize]
      (d-1-1) edge node[above]{$\alpha_1$} (d-1-2)
                   edge node[right]{$\varphi_1$} (d-2-1)
      (d-1-2) edge node[above]{$\alpha_2$} (d-1-3)
                   edge node[right]{$\varphi_2$} (d-2-2)
      (d-1-3) edge node[above]{$\alpha_{n - 2}$} (d-1-4)
      (d-1-4) edge node[right]{$\varphi_{n - 1}$} (d-2-4)
      (d-2-1) edge node[above]{$\beta_1$} (d-2-2)
      (d-2-2) edge node[above]{$\beta_2$} (d-2-3)
      (d-2-3) edge node[above]{$\beta_{n - 2}$} (d-2-4);
  \end{tikzpicture}
\end{center}
is called an \emph{$(n - 2)$-pushout diagram} if the mapping cone 
\begin{equation*}
  A_1\xrightarrow{\left[
      \begin{smallmatrix}
        -\alpha_1\\
        \hfill \varphi_1
      \end{smallmatrix}
    \right]} A_2 \oplus B_1 \xrightarrow{\left[
      \begin{smallmatrix}
        -\alpha_2 & 0\\
        \hfill \varphi_2 & \beta_1
      \end{smallmatrix}
    \right]} \cdots \xrightarrow{\left[
      \begin{smallmatrix} 
        -\alpha_{n - 2} & 0\\
        \hfill \varphi_{n - 2} & \beta_{n - 3}
      \end{smallmatrix}
    \right]} A_{n - 1} \oplus B_{n - 2} \xrightarrow{\left[
      \begin{smallmatrix}
        \hfill \varphi_{n - 1} & \beta_{n - 2}
      \end{smallmatrix}
    \right]} B_{n - 1}\to 0
\end{equation*}
is a right $(n - 2)$-exact sequence.

An \emph{$(n - 2)$-exact structure} on $\E$ is a class $\X$
of $(n - 2)$-exact sequences, whose members are called
\emph{admissible $(n - 2)$-exact sequences}, satisfying axioms similar
to those of exact categories. The pair $(\E,\X)$ is then
called an \emph{$(n - 2)$-exact category}.

An $(n - 2)$-exact category $(\E,\X)$ is \emph{Frobenius} if
the following properties are satisfied:
\begin{itemize}
\item For each $A\in\E$ there exist admissible $(n - 2)$-exact
  sequences
  \begin{equation*}
    0 \to A \to I_1 \to \cdots \to I_{n - 2} \to B \to 0  
  \end{equation*}
  and 
  \begin{equation*}
    0 \to C \to P_1 \to \cdots \to P_{n - 2} \to A \to 0
  \end{equation*}
  such that $I_1,\dots,I_{n - 2}$ are $\X$-injective and
  $P_1,\dots,P_{n - 2}$ are $\X$-projective (note that
  $\X$-projectives and $\X$-injectives are defined as usual).
\item The classes of $\X$-projectives and of $\X$-injectives coincide.
\end{itemize}

It is shown in \cite[Theorem 5.11]{J} that the stable category of a
Frobenius $(n - 2)$-exact category $\E$ has a natural $n$-angulation
induced by the admissible $(n - 2)$-sequences in $\E$. This motivates
the following definition.

\begin{definition}
  An $n$-angulated category is \emph{algebraic} if it is equivalent to
  the stable category of a Frobenius $(n - 2)$-exact category with its
  natural $n$-angulation.
\end{definition}

Let $\C$ be a category. When convenient, we denote by $d \cdot A$ the
$d$-fold multiple of the identity morphism of $A\in\C$. If $\C$ is an
$n$-angulated category, then for each $d\neq0$ and each $A\in \C$
there is an $n$-angle
\begin{equation*}
  A \xrightarrow{d\cdot 1} A \to (A/d)_1 \to \cdots \to (A/d)_{n - 2}
  \to \Sigma A.
\end{equation*}
Note that the complex $(A/d)_\bullet$ is well-defined up to homotopy
equivalence since $n$-angles are exact.

It is known that if $\C$ is an algebraic triangulated category, then
we have $d\cdot(A/d) = 0$ for all $d\neq0$ and for all $A\in \C$. This
is shown by Schwede in \cite[Proposition 1]{S1} by using the fact that
algebraic triangulated categories are tensored over
$\der^{\mathrm{b}}(\mod \Z)$. We now give an elementary proof of a
higher analog of this property.

\begin{proposition}
  \label{prop:n-order}
  Let $\E$ be a Frobenius $(n - 2)$-exact category. Then, for all
  $d\neq 0$ the morphism $d\cdot (A/d)_\bullet$ is null-homotopic as a
  morphism of complexes in the $n$-angulated category $\sE$.
\end{proposition}

\begin{proof}
  Firstly, by construction there is an $(n - 2)$-pushout diagram of
  admissible $(n - 2)$-exact sequences
  \begin{center}
    \begin{tikzpicture}
      \diagram{d}{2.5em}{2.5em}{
        I(A)_\bullet & 0 & A  & I_1 & \cdots & I_{n - 2} & \Sigma A & 0 \\
        B_\bullet & 0 & A & (A/d)_1 & \cdots & (A/d)_{n - 2} & \Sigma
        A & 0 \\ 
      };

      \path[->,midway,font=\scriptsize]
        (d-1-1) edge node[right]{$\varphi_\bullet$} (d-2-1)
        (d-1-2) edge (d-1-3)
        (d-1-3) edge node[above]{$\alpha_0$} (d-1-4)
                     edge node[right]{$d \cdot 1$} (d-2-3)
        (d-1-4) edge node[above]{$\alpha_1$} (d-1-5)
                     edge node[right]{$\varphi_1$} (d-2-4)
        (d-1-5) edge node[above]{$\alpha_{n - 3}$} (d-1-6)
        (d-1-6) edge node[above]{$\alpha_{n - 2}$}  (d-1-7)
                     edge node[right]{$\varphi_{n - 2}$}  (d-2-6)
        ([xshift=-0.025cm] d-1-7.south) edge[-] ([xshift=-0.025cm]
          d-2-7.north)
        ([xshift=0.025cm] d-1-7.south) edge[-] ([xshift=0.025cm]
          d-2-7.north)
        (d-1-7) edge (d-1-8)
        (d-2-2) edge (d-2-3)
        (d-2-3) edge node[above]{$\beta_0$} (d-2-4)
        (d-2-4) edge node[above]{$\beta_1$} (d-2-5)
        (d-2-5) edge node[above]{$\beta_{n - 3}$} (d-2-6)
        (d-2-6) edge node[above]{$\beta_{n - 2}$} (d-2-7)
        (d-2-7) edge (d-2-8);
    \end{tikzpicture}
  \end{center}
  which induces the $n$-angle
  \begin{equation*}
    A \xrightarrow{d\cdot 1} A \to (A/d)_1 \to \cdots \to (A/d)_{n -
      2} \to \Sigma A
  \end{equation*}
  in $\sE$, see \cite[Section~5]{J}.

  Secondly, by adding to the bottom the contractible $(n - 2)$-exact
  sequence
  \begin{equation*}
    0 \to 0 \to I_2 \to I_2 \oplus I_3 \to \cdots \to I_{n - 3} \oplus
    I_{n - 2} \to I_{n - 2} \to 0 \to 0,
  \end{equation*}
  we may assume that the above diagram is a good $(n - 2)$-pushout
  diagram in the sense of \cite[Defition-Proposition 2.14]{J}.  By the
  factorization property of good $(n - 2)$-pushout diagrams there is a
  commutative diagram
  \begin{center}
    \begin{tikzpicture}
      \diagram{d}{2.5em}{2.5em}{
        I(A)_\bullet & 0 & A  & I_1 & \cdots & I_{n - 2} & \Sigma A &
        0 \\
        B_\bullet & 0 & A & (A/d)_1 & \cdots & (A/d)_{n - 2} & \Sigma
        A & 0 \\
        I(A)_\bullet & 0 & A  & I_1 & \cdots & I_n & \Sigma A & 0 \\
      };

      \path[->,midway,font=\scriptsize]
        (d-1-1) edge node[right]{$\varphi_\bullet$} (d-2-1)
        (d-1-2) edge (d-1-3)
        (d-1-3) edge node[above]{$\alpha_0$} (d-1-4)
                     edge node[right]{$d \cdot 1$} (d-2-3)
        (d-1-4) edge node[above]{$\alpha_1$} (d-1-5)
                     edge node[right]{$\varphi_1$} (d-2-4)
        (d-1-5) edge node[above]{$\alpha_{n - 3}$} (d-1-6)
        (d-1-6) edge node[above]{$\alpha_{n - 2}$}  (d-1-7)
                     edge node[right]{$\varphi_{n - 2}$}  (d-2-6)
        ([xshift=-0.025cm] d-1-7.south) edge[-] ([xshift=-0.025cm]
          d-2-7.north)
        ([xshift=0.025cm] d-1-7.south) edge[-] ([xshift=0.025cm]
          d-2-7.north)
        (d-1-7) edge (d-1-8)
        (d-2-2) edge (d-2-3)
        (d-2-1) edge node[right]{$\psi_\bullet$} (d-3-1)
        (d-2-3) edge node[above]{$\beta_0$} (d-2-4)
        ([xshift=-0.025cm] d-2-3.south) edge[-] ([xshift=-0.025cm]
          d-3-3.north)
        ([xshift=0.025cm] d-2-3.south) edge[-] ([xshift=0.025cm]
          d-3-3.north)
        (d-2-4) edge node[above]{$\beta_1$} (d-2-5)
                     edge node[right]{$\psi_1$} (d-3-4)
        (d-2-5) edge node[above]{$\beta_{n - 3}$} (d-2-6)
        (d-2-6) edge node[above]{$\beta_{n - 2}$} (d-2-7)
                     edge node[right]{$\psi_{n - 2}$} (d-3-6)
        (d-2-7) edge[densely dashed] node[right]{$d \cdot 1$} (d-3-7)
        (d-2-7) edge (d-2-8)
        (d-3-2) edge (d-3-3)
        (d-3-3) edge node[above]{$\alpha_0$} (d-3-4)
        (d-3-4) edge node[above]{$\alpha_1$} (d-3-5)
        (d-3-5) edge node[above]{$\alpha_{n - 3}$} (d-3-6)
        (d-3-6) edge node[above]{$\alpha_{n - 2}$}  (d-3-7)
        (d-3-7) edge (d-3-8);
    \end{tikzpicture}
  \end{center}
  such that for all $k\in \{1,\dots,n - 2\}$ we have
  $\psi^k\circ\varphi^k = d\cdot 1$.  We only need to check that the
  bottom right square commutes. Indeed, we have
  \begin{align*}
   (d\cdot\beta_{n - 2})\circ\varphi_{n - 2} &= d\cdot \alpha_{n - 2},\\
   (\alpha_{n - 2}\circ\psi_{n - 2})\circ\varphi_{n - 2} &=d\cdot
     \alpha_{n - 2},\\
   (\alpha_{n - 2}\circ \psi_{n - 2})\circ \beta_{n - 3}  &= \alpha_{n
     - 2}\circ(\alpha_{n - 3}\circ\psi_{n - 3}) = 0,\intertext{and}
   (d\cdot\beta_{n - 2})\circ \beta_{n - 3}  &= 0.
  \end{align*}
  Since $[\varphi_{n - 2}\ \beta_{n-3}]\colon I_{n - 2} \oplus  
  (A/d)_{n - 3}\to (A/d)_{n - 2}$ is an epimorphism, the claim
  follows. Note that if $n = 3$ this shows that $d\cdot (A/d) = 0$ (by
  convention, $(A/d)_0 \ceq A$). In the rest of the proof we assume
  $n\geq 4$.

  Thirdly, we deuce from the commutative diagram
  \begin{center}
    \begin{tikzpicture}
      \diagram{d}{2.5em}{2.5em}{
        B_\bullet & 0 & A & (A/d)_1 & \cdots & (A/d)_{n - 2} & \Sigma
        A & 0 \\
        I(A)_\bullet & 0 & A  & I_1 & \cdots & I_{n - 2} & \Sigma A &
        0 \\
        B_\bullet & 0 & A & (A/d)_1 & \cdots & (A/d)_{n - 2} & \Sigma
        A & 0\\
      };

      \path[->,midway,font=\scriptsize]
        (d-1-2) edge (d-1-3)
        (d-1-1) edge node[right]{$\psi_\bullet$} (d-2-1)
        (d-1-3) edge node[above]{$\beta_0$} (d-1-4)
        ([xshift=-0.025cm] d-1-3.south) edge[-] ([xshift=-0.025cm]
          d-2-3.north)
        ([xshift=0.025cm] d-1-3.south) edge[-] ([xshift=0.025cm]
          d-2-3.north)
        (d-1-4) edge node[above]{$\beta_1$} (d-1-5)
                     edge node[right]{$\psi_1$} (d-2-4)
        (d-1-5) edge node[above]{$\beta_{n - 3}$} (d-1-6)
        (d-1-6) edge node[above]{$\beta_{n - 2}$} (d-1-7)
                     edge node[right]{$\psi_{n - 2}$} (d-2-6)
        (d-1-7) edge node[right]{$d \cdot 1$} (d-2-7)
        (d-1-7) edge (d-1-8)
        (d-2-1) edge node[right]{$\varphi_\bullet$} (d-3-1)
        (d-2-2) edge (d-2-3)
        (d-2-3) edge node[above]{$\alpha_0$} (d-2-4)
                     edge node[right]{$d \cdot 1$} (d-3-3)
        (d-2-4) edge node[above]{$\alpha_1$} (d-2-5)
                     edge node[right]{$\varphi_1$} (d-3-4)
        (d-2-5) edge node[above]{$\alpha_{n - 3}$} (d-2-6)
        (d-2-6) edge node[above]{$\alpha_{n - 2}$}  (d-2-7)
                     edge node[right]{$\varphi_{n - 2}$}  (d-3-6)
        ([xshift=-0.025cm] d-2-7.south) edge[-] ([xshift=-0.025cm]
          d-3-7.north)
        ([xshift=0.025cm] d-2-7.south) edge[-] ([xshift=0.025cm]
          d-3-7.north)
        (d-2-7) edge (d-2-8)
        (d-3-2) edge (d-3-3)
        (d-3-3) edge node[above]{$\beta_0$} (d-3-4)
        (d-3-4) edge node[above]{$\beta_1$} (d-3-5)
        (d-3-5) edge node[above]{$\beta_{n - 3}$} (d-3-6)
        (d-3-6) edge node[above]{$\beta_{n - 2}$} (d-3-7)
        (d-3-7) edge (d-3-8);
    \end{tikzpicture}
  \end{center}
  and the dual of \cite[Comparison Lemma 2.1]{J} that there exist a
  homotopy $h_\bullet\colon d\cdot B_\bullet\to\varphi_\bullet \circ
  \psi_\bullet$ with $h_{n - 1}\colon \Sigma A\to (A/d)_{n - 2}$ equal
    to 0.  Next, observe that $h_1 \circ \beta_0 = 0$, hence $h_1$ factors
  through $\beta_1$ as $\gamma \circ \beta_1$, say. Thus, by replacing
  $h_2$ by $h_2 + \beta_0 \circ \gamma$, we may assume that $h_1 =
  0$. Therefore $h_\bullet$ induces a homotopy $d\cdot (A/d)_\bullet
  \to \varphi_\bullet \circ\psi_\bullet$. This shows that $d\cdot
  (A/d)_\bullet$ is null-homotopic as a morphism of complexes in
  $\sE$.
\end{proof}

We now prove the main result of this section.

\begin{theorem}
  \label{thm:not_algebraic}
  Let $R$ be a commutative local ring with principal maximal ideal $\m
  = (p) \neq 0$ such that $\m^2 = 0$ and $\C$ the category of finitely
  generated free $R$-modules. Suppose moreover that $2p=0$. If $n$ is
  odd and there exists $d\in\Z$ such that $d\cdot 1_R\in \m \setminus
  \{0\}$, then the $n$-angulated category $(\C,\Sigma,\nang)$
  constructed in Theorem \ref{thm:main} is \emph{not} algebraic.
\end{theorem}

\begin{proof}
  Let $d\in \Z$ be such that $d\cdot 1_R \in\m\setminus\{0\}$. Hence
  $d\cdot 1_R = up$ from some unit $u$ in $R$. The case $n = 3$ is
  clear since $d\cdot 1_R = up\neq 0$.

  Let $n\geq 4$. Since $up$ is not a unit, every $n$-angle having
  $up\colon R\to R$ as first morphism is isomorphic to the $n$-angle
  \begin{equation*}
    R\xrightarrow{p} R\xrightarrow{p} R\xrightarrow{p} \cdots
    \xrightarrow{p} R\xrightarrow{p}R.
  \end{equation*}
  Therefore, up to a contractible summand, we have
  \begin{equation*}
    (R/d)_\bullet = (R\xrightarrow{p} \cdots
    \xrightarrow{p} R).
  \end{equation*}
  Suppose that there exists a null-homotopy of $d\cdot (R/d)_\bullet =
  (up,\dots,up)$. Thus, there exist $q_1,\dots,q_{n-3}\in R$ such that
  \begin{equation*}
    up = pq_1 = q_1p + pq_2 = \cdots = q_{n - 4}p + pq_{n - 3} = q_{n
      - 3}p.
  \end{equation*}
  Given that $2p = 0$ and that $n$ is odd, by adding the above
  equalities we have $(n - 2)up = up = 0$, a contradiction. Hence the
  $n$-angulated category $(\C,\Sigma,\nang)$ is not algebraic by
  Proposition \ref{prop:n-order}.
\end{proof}

\begin{remark}
  Note that if $n$ is even, then the sequence
  $(u,0,u,\dots,0,u)\in R^{n - 3}$ gives a null-homotopy of
  $d\cdot(R/d)_\bullet = (up,\dots,up)$ in the setting of Theorem
  \ref{thm:not_algebraic}. Hence we cannot deduce from Proposition
  \ref{prop:n-order} that the $n$-angulated category
  $(\C,\Sigma,\nang)$ is not algebraic in this case.
\end{remark}

\begin{example}
  Let $R$ be the ring $\Z/(4)$. Then, the assumptions of Theorem
  \ref{thm:not_algebraic} are satisfied since we have $0\neq2\in
  (2)$. Thus, for odd $n$, the corresponding $n$-angulated category is
  not algebraic.
\end{example}

%
%
\section{Acknowledgements}

We thank Amnon Neeman for pointing out Remark \ref{rem:Neeman}, and an
anonymous referee for major improvements in the paper.

%
%

%
%
\end{document}